\documentclass[12pt,a4paper]{amsart}

\usepackage{a4wide}
\RequirePackage[OT1]{fontenc}
\RequirePackage{amsthm,amsmath}
\RequirePackage[colorlinks,citecolor=blue,urlcolor=blue]{hyperref}

\RequirePackage{ifthen}
\newboolean{extendedversion}
\setboolean{extendedversion}{true}

\usepackage[authoryear]{natbib}

\numberwithin{equation}{section}
\theoremstyle{plain}

\theoremstyle{remark}

\newtheorem{theorem}{Theorem}[section]
\newtheorem{corollary}{Corollary}[section]

\newtheorem{remark}{Remark}[section]
\newtheorem{lemma}{Lemma}[section]

\newcommand{\N}{ \mathbb{N} }

\newcommand{\Z}{ \mathbb{Z} }

\newcommand{\R}{ \mathbb{R} }

\newcommand{\trunc}[1]{ {\lfloor #1 \rfloor} }
\newcommand{\wh}[1]{ \widehat{ #1 } }
\newcommand{\wt}[1]{ \widetilde{ #1 } }

\newcommand{\calD}{\mathcal{D}}

\newcommand{\calF}{\mathcal{F}}
\newcommand{\calG}{\mathcal{G}}

\newcommand{\calI}{\mathcal{I}}
\newcommand{\calJ}{\mathcal{J}}

\newcommand{\calM}{\mathcal{M}}
\newcommand{\calN}{\mathcal{N}}

\newcommand{\calU}{\mathcal{U}}

\newcommand{\calX}{\mathcal{X}}
\newcommand{\calY}{\mathcal{Y}}

\newcommand{\eins}{\mathbf{1}}
\newcommand{\matA}{{ \mathbf A}}

\newcommand{\matC}{{\mathbf C}}

\newcommand{\matP}{{ \mathbf P}}

\newcommand{\matV}{{\mathbf V}}
\newcommand{\matW}{{\mathbf W}}

\newcommand{\vecnull}{{ 0}}

\newcommand{\vecB}{{ \mathbf B}}
\newcommand{\vecc}{{ \mathbf c}}
\newcommand{\vecD}{{ \mathbf D}}

\newcommand{\vecu}{{ \mathbf u}}
\newcommand{\vecv}{{ \mathbf v}}
\newcommand{\vecr}{{ \mathbf r}}

\newcommand{\vecR}{{ \mathbf R}}
\newcommand{\vecS}{{ \mathbf S}}

\newcommand{\vecw}{{ \mathbf w}}

\newcommand{\vecX}{{ \mathbf X}}

\newcommand{\vecY}{{ \mathbf Y}}
\newcommand{\vecZ}{{ \mathbf Z}}

\newcommand{\vecV}{{ \mathbf V}}

\newcommand{\bfmu}{\mu}
\newcommand{\bfbeta}{\beta}

\newcommand{\bfxi}{\boldsymbol{\xi}}

\newcommand{\bfSigma}{\boldsymbol{\Sigma}}

\newcommand{\Var}{{\mbox{Var\,}}}
\newcommand{\Cov}{{\mbox{Cov\,}}}

\newcommand{\matid}{I}

\newcommand{\trace}{ \operatorname{tr} }

\newcommand{\id}{\operatorname{id}}

\newcommand{\cadlag}{c\`adl\`ag }

\begin{document}

\title[Inference for Large Covariance Matrices]{Large-sample approximations for variance-covariance matrices of high-dimensional time series}

\author{Ansgar Steland}
\address{Institute of Statistics\\ RWTH Aachen University\\W\"ullnerstr. 3, D-52056 Aachen, Germany}
\email{steland@stochastik.rwth-aachen.de}

\author{Rainer von Sachs}
\address{Institut de Statistique, Biostatistique et Sciences Actuarielles (ISBA)\\ Universit\'e catholique de Louvain\\ Voie du Roman Pays 20, B-1348 Louvain-la-Neuve, Belgium}
\email{rainer.vonsachs@uclouvain.be}

\date{December 2015. This is a preprint version of an article to appear in {\em Bernoulli}}

\maketitle

\begin{abstract}
Distributional approximations of (bi--) linear functions of sample variance-covariance matrices play a critical role to analyze vector time series, as they are needed for various purposes, especially to draw inference on the dependence structure in terms of second moments and to analyze projections onto lower dimensional spaces as those generated by principal components. This particularly applies to the high-dimensional case, where the dimension $d$ is allowed to grow with the sample size $n$ and may even be larger than $n$. We establish large-sample approximations for such bilinear forms related to the sample variance-covariance matrix of a high-dimensional vector time series in terms of strong approximations by Brownian motions. The results cover weakly dependent as well as many long-range dependent linear processes and are valid for  uniformly $ \ell_1 $-bounded projection vectors, which arise, either naturally or by construction, in many statistical problems extensively studied for high-dimensional series.
Among those problems are sparse financial portfolio selection, sparse principal components, the LASSO, shrinkage estimation and change-point analysis for high--dimensional time series, which matter for the analysis of big data and are discussed in greater detail. \\
Keywords:
{Brownian motion} {Central limit theorem} {Change-points} {Long memory}
{Multivariate analysis} {Principal component analysis} {Portfolio analysis} {Strong approximation}
{Time Series}
\end{abstract}

\section{Introduction}

The estimation of high-dimensional variance-covariance matrices based on a vector time series arises in diverse areas such as financial portfolio optimization, image analysis and multivariate time series analysis in general. Of particular interest is the case that the dimension $ d = d_n $ of the time series grows even faster than the sample size $n$. Due to the lack of consistency of the sample variance-covariance estimator with respect to commonly used norms such as the Frobenius norm, various regularized modifications have been proposed and extensively studied within a high-dimensional context. For example, banding and tapering estimators, recently studied by \cite{BickelLevina2008} for Gaussian samples, may achieve consistency if $ \log d_n/n = o(1) $. \cite{ChenXuWu2014} establish bounds for thresholded sample covariance estimators for a high-dimensional vector time series in terms of scaled Frobenius and spectral norms, allowing for non-stationarity and dependence. The performance of shrinkage estimation, a widely used technique dating back to the seminal work of \cite{Stein1956}, has been investigated by \cite{LedoitWolf2003} for i.i.d. samples of growing dimension and further studied for the weak dependent case in \cite{Sancetta2008}. For results on shrinkage estimation in the frequency domain we refer to \cite{BoehmSachs2009}.

However, often the estimation of the $d_n^2$-dimensional variance-covariance matrix is an intermediate step and one is mainly interested in the behavior of functions of the sample variance-covariance matrix, especially quadratic and bilinear forms which naturally arise when studying projection type statistics. In addition, one often needs distributional approximations of such functions, in order to construct statistical decision procedures. Whereas consistency and performance properties have been already investigated to some extent (see above), there are only a few results about the asymptotic distribution theory in the sense of distributional convergence (weak convergence) and strong approximations by Brownian motions, respectively, going beyond the classical
 results. Our framework allows to embed autocovariances and, in particular, cross-autocovariances as well.
To the best of our knowledge, there are no weak convergence results addressing those issues within a high-dimensional framework, i.e. for a large number of correlated time series. In some sense close to the present paper are the following results for fixed dimension and autocovariance matrices.. \cite{WuMin2005} derived a CLT for a finite number of sample auto-covariances assuming a linear process. For a one-dimensional context, \cite{XiaoWu2014} study a central limit theorem (CLT), Portmanteau tests and simultaneous inference for a growing number of lags based on a Gumbel type extreme value theory, see also the review \cite{Wu2011} and \cite{Jirak2011}. In \cite{WuHuangZheng2010} the estimation of autocovariances for long memory linear processes has been discussed and studied in depth including the case of a finite number of lags starting at a large lag $k_n $ with $ k_n/n = o(1) $.  \cite{Kouritzin1995}, also working within a linear process framework, established large-sample distributional asymptotics, based on strong approximations, of the sample cross-covariance matrix for two time series. His assumptions are weak enough to cover the case of long-range dependence as well.

The present paper builds upon the latter result by establishing strong approximations of bilinear forms associated to the centered sample covariance matrix of a high-dimensional vector time series. The result implies the validity of a central limit theorem (CLT) for scaled bilinear forms
 $ \sqrt{n} \vecv_n'[ \wh{\bfSigma}_n - E( \wh{\bfSigma} ) ] \vecw_n $, where $ \wh{\bfSigma}_n $ is the usual sample variance-covariance matrix and $ \vecv_n $, $ \vecw_n $ are weighting vectors. It turns out that  $ d_n $ may even grow faster than $n$.

Concerning the weighting (or projection) vectors, our results assume that they are uniformly bounded in the $ \ell_1 $-norm. Such projections naturally arise in many problems studied in the area of high-dimensional statistics and probability: Sparse optimal portfolio selection, as recently studied by \cite{BrodieEtAl2009}, deals with explicit construction of $ \ell_1 $-bounded portfolios from historical data sets. The same applies to several approaches of sparse principal component analysis, especially those of \cite{JollifeEtAl2003}, \cite{ShenHuan2008} and \cite{WittenTibshirani2009}, where $ \ell_1 $-bounded principal components are constructed, in order to represent high-dimensional data by only a few sparse projections. We discuss those applications in greater detail in Section~\ref{Applications}. We also illustrate how the results can be applied to obtain distributional approximations of shrinkage estimators of a high-dimensional covariance matrix. Lastly, we discuss the application to detect the presence of a change-point. Such procedures analyze the data to identify changes in the distribution and have been thoroughly studied for various second order problems for time series. Of course, a change in a covariance $ \gamma_{X}(i,j) = E(X^{(i)} X^{(j)}) $, $ t \ge 1 $, of a vector time series can be analyzed by applying any method which is sensitive to location changes to the sequence $ X_t^{(i)} X_t^{(j)} $, $ t \ge 1 $. Such methods are discussed in  \cite{HuskovaHlavka2012}, \cite{HorvathAue2013}. Detectors based on local linear estimators have been proposed by \cite{Steland2010} and kernel detectors were studied by \cite{Steland2004} and \cite{Steland2005}. For methods based on characteristic functions we refer to \cite{StelandRafalowicz2014} and the references therein. 

The organization of the paper is as follows.  Section~\ref{Projections} explains the general setting, discusses its basic relationship to projection-based analyses and introduces the bilinear form of interest. Notation, the specific model for the vector time series and its interpretation in terms of an infinite--dimensional  latent factor model as well as assumptions are introduced and discussed in Section~\ref{ModelAssumptions}.
Section~\ref{Asymptotics} provides several results on strong approximations, which imply CLTs and functional central limit theorems (FCLTs) in the sense of Donsker's theorem. We also propose  estimators for the asymptotic variance parameters and show their consistency, uniformly in the dimension. Lastly, Section~\ref{Applications} elaborates on several statistical problems to which our results are directly applicable. Proofs of the main results are provided in an appendix.

\section{Projection-based analysis of high-dimensional time series}
\label{Projections}

Let us assume that we observe $ d $ possibly dependent time series such that at time $n$ we are given the observations
\[
  Y_1^{(\nu)}, \dots, Y_n^{(\nu)}, \qquad \nu = 1, \dots, d_n,
\]
where the dimension $ d = d_n $ may grow with the sample size $n$, such that, as time proceeds, there may be more and more time series available. Equivalently, we are
given a time series of length $n$ of possibly dependent random vectors
\[
 \vecY_{ni} = (Y_i^{(1)}, \dots, Y_i^{(d_n)} )',\qquad 1 \le i \le n,
\]
of dimension $ d_n $, constituting the $ (n \times d_n) $--dimensional data matrix
$
  \calY_n = \left( Y_i^{(j)} \right)_{1 \le i \le n, 1 \le j \le d_n}.
$
We are interested in the second moment structure and thus assume $ E(Y_i^{(j)} ) = 0 $ for all $ j = 1, \dots, d_n $, $ i = 1, \dots, n $ and $ n \ge 1$.
Our assumptions on the coordinate processes, basically that they are linear processes with sufficiently fast decreasing coefficients, are weak enough to cover the common framework of correlated ARMA($p,q$)-processes and also allow for a wide class of long-range dependent series. 

Let us assume for a moment that $ \vecY_{n1}, \dots, \vecY_{nn} $ is stationary, a condition that we shall relax later, and let 
\[
  \vecY_n = ( Y^{(1)}, \dots, Y^{(d_n)} )' 
\] 
be a generic copy. For the analysis of such high-dimensional time series, the unknown variance-covariance matrix
\[
  \bfSigma_n = E( \vecY_n \vecY_n' )
   = \left( E( Y^{(\nu)} Y^{(\mu)} ) \right)_{1 \le \nu, \mu \le d_n}
\]
is of substantial interest, but difficult to estimate from past data, in particular if $ d_n >> n $. It comprises the second-order information on the dependence structure of the $d_n$ variables. Any conclusions on the correlation structure have to rely on estimators calculated from the time series, and inferential procedures require appropriate large-sample asymptotics. Let
\begin{equation}
\label{DefSigmaHat}
  \wh{\bfSigma}_n = \frac1n \sum_{i=1}^n \vecY_{ni} \vecY_{ni}'
\end{equation}
be the $ (d_n \times d_n) $-dimensional sample variance--covariance matrix
with elements
\[
 \wh{\sigma}_{\nu\mu} = \frac{1}{n} \sum_{i=1}^n Y_i^{(\nu)} Y_i^{(\mu)},
 \qquad \nu, \mu = 1, \dots, d_n.
\]
Before proceeding, let us observe  that the above framework also covers the case of a univariate time series $ \{ Z_k : k \ge 0 \} $ as an interesting special case. The {\em embedding} is given by
\[
  \vecY_{ni} = (Z_i, Z_{i+1}, \dots, Z_{i+d_n-1})', \qquad i = 1, \dots, n.
\]
Then
\[
  \wh{\sigma}_{\nu\mu} = \frac1n \sum_{i=1}^n Z_{i+\nu-1} Z_{i+\mu-1},
  \qquad 1 \le \nu, \mu \le d_n.
\]
It follows that the $(h+1)$th element of the first row of $ \wh{\bfSigma}_n $ estimates
$ \gamma_Z(h) = E(Z_0Z_h) $ using the observations $ Z_1, \dots, Z_T$, where $ T = n+h $, and can be written as 
\[
  \wt{\gamma}_Z(h) = \frac{T}{T-h} \wh{\gamma}_Z(h), \qquad \wh{\gamma}_Z(h) = \frac{1}{T} \sum_{i=1}^{T-h} Z_i Z_{i+h},
\]
for $ h = 0, \dots, T-1 $.
In a similar way, one may consider autocovariances and cross-covariances of, say, $ r_n $ time series $ \{ Z_k^{(l)} : k \ge 0 \} $, $ l = 1, \dots, r_n $.

Estimators of $ \bfSigma_n $ are also needed and have to be evaluated in terms of their asymptotic laws when interest focuses on the analysis of (a set of) linear combinations of $ \vecY_n $. Here $ \vecY_n $ may be a generic copy when the vector time series is strictly stationary or equal (in distribution) to $ \vecY_{nn} $ in the general case. Typical examples are convex combinations, contrasts and, more generally, projections. Thus let
\[
  \vecw_n = (w_{1}, \dots, w_{d_n} )', \qquad n \ge 1,
\]
be a sequence of weights $ w_j = w_{nj} $, not necessarily non-negative, with uniformly bounded $ \ell_1 $-norm, i.e.
\begin{equation}
\label{l1Condition}
  \sup_{n \in \N} \| \vecw_n \|_{\ell_1} = \sup_{n \in \N} \sum_{\nu=1}^{d_n} | w_\nu | < \infty
\end{equation}
The class of weighting vectors of the form $ \vecw_n = ( w_1, w_2, \dots, w_{d_n} )' $ for some sequence $ \{ w_i : i \in \N \} $ with $ \sum_{i=1}^\infty | w_i | < \infty $ certainly satisfies the assumptions and covers, for example,
the case of averaging a finite number of coordinates. However, if $ w_i > 0 $ holds for infinitely many $ i \ge 1 $,
then $ w_i = o(1) $, as $ i \to \infty $. Hence infinitely many coordinates are not really taken into 
account. Allowing for weights that depend on the dimension $d_n$ substantially widens the scope.
Now, for example, one may average all coordinates by using the weights $ w_{ni} = 1/d_n $, for $ i = 1, \dots, d_n$.
Although the dependence on the sample size $n$ through the dimension $ d_n $ may be of primary importance for high-dimensional problems, several of our results even allow the weights to depend on $n$.

The variance of the projection $ \vecw_n' \vecY_n $ is given by $ \vecw_n' \bfSigma_n \vecw_n $ and can be estimated nonparametrically by the quadratic form
$
  Q_n( \vecw_n ) = \vecw_n' \widehat{\bfSigma}_n \vecw_n,
$
whose random fluctuations may severely affect any inferential procedure related to $ \vecw_n'\vecY_n$. More generally, we shall study the bilinear form
\begin{equation}
\label{DefBiLinearForm}
  Q_n( \vecv_n, \vecw_n ) = \vecv_n' \wh{\bfSigma}_n \vecw_n
\end{equation}
for weighting vectors $ \vecv_n $ and $ \vecw_n $ with uniformly bounded $\ell_1 $-norms in the sense of (\ref{l1Condition}), which corresponds to the estimator of the covariance of two projections $ \vecv_n' \vecY_n $ and $ \vecw_n' \vecY_n $. Observe that (\ref{DefBiLinearForm}) also allows us to handle the case of
weighted sums of subsets $ \wh{\bfSigma}_{\calI, \calJ} = \{ \wh{\sigma}_{ij} : i \in \calI, j \in \calJ \} $, where $ \calI, \calJ \subset \{1, \dots, d_n \} $. 

It is worth mentioning that $ Q_n $ remains bounded even for degenerate covariance matrices, if $ \vecv_n $ and $ \vecw_n $ have uniformly bounded $ \ell_1 $-norm, as can be seen if we put $ \bfSigma_n = \sigma \eins \eins' $, where here and throughout the article $ \eins = (1, \dots, 1)' \in \R^{d_n} $, leading to
\[
  | Q_n( \vecv_n, \vecw_n ) | = \sigma | \vecv_n' \eins \eins' \vecw_n | = \sigma
  | \sum_i v_{ni} \sum_i w_{ni} | \le \sigma \| \vecv_n \|_{\ell_1} \| \vecw_n \|_{\ell_1}.
\]
Thus, the $ \ell_1 $-norm condition is a natural one: It ensures that $ Q_n $ maps products $ \calU_\delta \times \calU_\delta $ of $ \delta $-balls 
\[
  \calU_\delta = \{ \{ \vecv_n \} : \sup_{n \in \N} \| \vecv_n \|_{\ell_1} \le \delta \},
\]
for $ \delta > 0 $, onto bounded sets, for all covariance matrices with uniformly bounded entries, thus
including cases that correspond to perfectly correlated coordinates.

The behavior of projections for high-dimensional observations has also been studied by
\cite{DiaconisFreedman1984}, but from a different perspective. There it is shown that for large dimension $d$ projections
$ \vecw'\vecX $ are asymptotically normal under weak assumptions by
showing that, given a (non-random) sample $ X_1, \dots, X_n $ and a unit vector
$ \vecw $ uniformly distributed on the $d$-dimensional unit sphere, the empirical
measure of the sample $ \vecw'\vecX_1, \dots, \vecw'\vecX_n $ converges weakly
to a normal law, in probability. Further, 
the joint distribution of two linear combinations, say, $ \vecw_d'\vecZ $ and $ \vecv_d'\vecZ $, of a $d$-dimensional  random vector $ \vecZ $
possessing a Lebesgue density and being standardized, i.e. $ E(\vecZ ) = 0 $ and $ E( \vecZ\vecZ' ) = \id_d $, is bivariate normal in that sense  with unit variances and covariance $ \vecw'\vecv $. However, in that work the projection vectors are 
{\em random} and the data are assumed {\em fixed} (e.g. by conditioning) and constrained to satisfy conditions that are satisfied by, e.g., i.i.d. random vectors of dimension $d_n $ with i.i.d. entries. Contrary, in this paper the projection vectors are fixed and the observations random. Assuming a linear process framework, we provide Gaussian approximations for the sample estimates of the variances of the projections.

\section{A framework for high-dimensional time series}
\label{ModelAssumptions}

Our results dealing with strong approximations of the bilinear forms introduced above rely on some linear process framework. This section provides a careful introduction, discusses some interesting properties and interpretations as well as introduces required notation and some preparatory approximations used later.

\subsection{Model and assumptions}

Let $ \{ \epsilon_k : k \in \Z \} $ be a sequence of independent random variables with mean zero, variances 
\[
  \sigma_k^2 = E( \epsilon_k^2 )
\]
and uniformly bounded moments of the order $ (4+\delta)$,
\begin{equation}
\label{BoundedFourthMoments}
  \sup_k E | \epsilon_k |^{4+\delta} < \infty,
\end{equation}
for some $ \delta > 0 $, such that $ \gamma_k = E \epsilon_k^4 $ and $ \sigma_k^2 $ are finite, for all $k \in \Z $.

We assume that the $\nu$th coordinate of $ \vecY_n $ is given by
\begin{equation}
\label{SubordModel}
  Y_k^{(\nu)} = Y_{nk}^{(\nu)} = \sum_{j=0}^\infty c_{nj}^{(\nu)} \epsilon_{k-j}, \qquad k = 1, \dots, n,
\end{equation}
for coefficients $ \{ c_{nj}^{(\nu)} : j  \in \N_0 \} $, $ \nu = 1, \dots, d_n $. We mainly have in mind the case that we observe, at time $n$, the first $n$ observations of $ d_n $ {\em sequences} $ \{ Y_k^{(\nu)} :  k \ge 0 \} $, $ \nu = 1, \dots, d_n $, but our results also allow for {\em arrays} $ \{ Y_{nk}^{(\nu)} : k \ge 0, n \ge 1 \} $, since the coefficients may depend on $n$. Also notice that in model (\ref{SubordModel}) $ Y_{nk}^{(\nu)} $ is well defined for $ k > n $.

Model (\ref{SubordModel}) implies that the cross-sectional as well as serial correlations have a specific structure, since
\begin{align}
\label{CovEq1}
 \Cov( Y_{nt}^{(\nu)}, Y_{nt}^{(\mu)} ) 
 &= \sum_{j=0}^\infty c_{nj}^{(\nu)} c_{nj}^{(\mu)} \sigma_{t-j}^2 \\
\label{CovEq2}
 \Cov( Y_{nt}^{(\nu)}, Y_{n,t+h}^{(\mu)} ) &= \sum_{j=0}^\infty c_{nj}^{(\nu)} c_{n,j+h}^{(\mu)} \sigma_{t-j}^2,
\end{align} 
for $ h > 0 $, $ 1 \le \nu, \mu \le d_n $ and all $t$.  Consequently, the cross-sectional variance-covariance matrix $ \Var( \vecY_{nt} ) $ is given by 
\begin{equation}
\label{SigmaSeries}
  \bfSigma_n[t] = \matC_n \Lambda \matC_n'
  = \sum_{j=0}^\infty \sigma_{t-j}^2 \vecc_{nj} \vecc_{nj}',  \quad \Lambda = \operatorname{diag}( \sigma_0^2, \sigma_1^2, \dots ),
\end{equation}
where $ \matC_n = ( c_{nj}^{(\nu)} )_{1 \le \nu \le d_n, 1 \le j} $  is the $ (d_n \times \infty) $-dimensional matrix with column vectors 
$
  \vecc_{nj} = (c_{nj}^{(1)}, \dots, c_{nj}^{(d_n)})' $, $ j \ge 0.
$
The lag $h$ serial covariance matrix attains the representation
\[ 
  \bfSigma_n(h) = E( \vecY_{nt} \vecY_{n,t+h}' ) = \matC_n \Lambda (L^{-h} \matC_n)'
\] 
where $L$ denotes the lag operator that acts on all columns, i.e. $ L^{-1} \matC_n = ( c_{n,i+1}^{(\nu)} )_{i \ge 0, 1 \le  \nu \le d_n} $.

We shall impose the following condition on the decay of the coefficients $ c_{nj}^{(\nu)} $, which is similar to the assumption imposed in \cite{JohnstoneLu2009}, where it controls the principal component eigenvectors within a factor model  (also see our discussion below), and to condition (2.4) in \cite{ChanHorvathHuskova2013}, where it controls the error terms  of a panel time series model.
 
\textbf{Assumption (A)} The sequences $ \{ c_{nj}^{(\nu)} : j  \in \N_0 \} $ satisfy
\begin{equation}
\label{SuffCond}
  \sup_{n \in \N} \max_{1 \le \nu \le d_n} | c_{nj}^{(\nu)} |^2 << (j \vee 1)^{-3/2-\theta}
\end{equation}
for some $ 0 < \theta < 1/2 $. 

Here and in the sequel $ a_n << b_n $ stands for $ a_n = O(b_n) $. Further, we shall write $ a_{nm} \stackrel{n,m}{<<} b_{nm} $ if there exists a constant $ C $ such that $ a_{nm} \le C b_{nm} $ for all $n, m$ 

Indeed, (\ref{SuffCond}) covers not only short memory processes for which the covariances, say, $ r_k = E(X_0X_k) $, are summable, i.e. $ \sum_k |r_k| < \infty $, but
also many long-range dependent series. An example for the latter is fractionally integrated noise of order $ d \in (-1/2, 1/4-\theta/2 ) $, i.e. a stationary solution of the equation 
\[
  (1-L)^d X_t = \epsilon_t, 
\]
where $ \{ \epsilon_t \} $ is a white-noise series, that is given by
$ X_t = \sum_{k=0}^\infty \theta_k \epsilon_{t-k} $ with coefficients
$ \theta_k = \Gamma(k+d)/(\Gamma(k+1) \Gamma(d)) \sim k^{d-1}/\Gamma(d) $,
see e.g. \cite{Steland2012}. The growth condition $ O( j^{-3/4-\theta} ) $ for some $ \theta > 0 $ on the coefficients of the linear processes arises also in other works, especially in \cite{WuHuangZheng2010} where the asymptotics of sample autocovariances of a linear process is studied. The case that the coefficients are $ O( j^{-3/4} L(j) ) $ for some slowly-varying function $L$ represents a boundary case. Here one can obtain Gaussian limits for sample autocovariances, see \cite[Theorem~3]{WuHuangZheng2010} and \cite{Hosking1996} for i.i.d. Gaussian innovations, but then the convergence rate changes from $ n^{-1/2} $ to $ (n/\bar{L}_n)^{-1/2}$ where $ \bar{L}_n = \sum_{i=1}^n L^4(i)/i $.



\ifthenelse{\boolean{extendedversion}}{
It is worth noting that, under certain circumstances, one may interpret model (\ref{SubordModel}) as an infinite-dimensional latent one-factor model. For that purpose,
let us assume for a moment that $ \vecY_1, \dots, \vecY_n $ is a stationary series, such that $ \sigma_k^2 = \sigma^2 $ for all $k$. Then, by stationarity, $ \vecY_n = (Y^{(1)}, \dots, Y^{(d_n)} )' $ satisfies 
\begin{align}
\label{EqYn}
  \vecY_n & \stackrel{d}{=} \matC_n \boldsymbol{\epsilon}_n, \quad \boldsymbol{\epsilon}_n = (\epsilon_0, \epsilon_{-1}, \dots), \\
\label{EqSigma}
  \bfSigma_n &= \sigma^2 \matC_n \matC_n', 
\end{align}
where $ \{ \epsilon_n \} $ is the unobservable common factor. 
}{}

\subsection{Some preparatory approximations}

In the main proof we shall study in detail the linear process
$ \sum_{j=0}^\infty c_j^w \epsilon_{k-j} $, $ k \ge 1$, with coefficients
\begin{equation}
\label{DefCw}
  c_j^w = \sum_{\nu=1}^{d_n} w_\nu c_j^{(\nu)}, \qquad j \ge 0,
\end{equation}
associated to a weighting vector $ \vecw_n $. Behind our main results are martingale approximations related to that linear process, whose definitions require the following quantities, which are controlled under Assumption (A) by virtue of Lemma~\ref{TechnicalLemma} given below. Let
\begin{equation}
  f_{0,j}^{(n)} = f_{0,j}^{(n)}( \vecv_n, \vecw_n ) 
  = \sum_{\nu, \mu = 1}^{d_n} v_\nu w_\mu c_j^{(\nu)} c_j^{(\mu)},  \qquad j = 0, 1, \dots,
\end{equation}
\begin{equation}
  f_{l,j}^{(n)} = f_{l,j}^{(n)}( \vecv_n, \vecw_n ) 
  = \sum_{\nu, \mu = 1}^{d_n} v_\nu w_\mu [ c_j^{(\nu)} c_{j+l}^{(\mu)} + c_j^{(\mu)} c_{j+l}^{(\nu)} ], \qquad l = 1, 2, \dots; \ j = 0, 1, \dots,
\end{equation}
and
\begin{equation}
  \wt{f}_{l,i}^{(n)} = \wt{f}_{l,i}^{(n)}( \vecv_n, \vecw_n )
  = \sum_{j=i}^{\infty} f_{l,j}^{(n)} 
  = \sum_{j=i}^\infty \sum_{\nu, \mu = 1}^{d_n} v_\nu w_\mu [c_j^{(\nu)} c_{j+l}^{(\mu)} + c_j^{(\mu)} c_{j+l}^{(\nu)} ], \qquad l, i  = 0, 1, \dots.
\end{equation}
Clearly, the quantities $ \wt{f}_{l,i}^{(n)} $ have the scaling property
\begin{equation}
\label{FTildeScaling}
\wt{f}_{l,i}^{(n)}( s_1 \vecv_n, s_2 \vecw_n ) =
  s_1 s_2 \wt{f}_{l,i}^{(n)}( \vecv_n, \vecw_n )
\end{equation}
for arbitrary $ s_1, s_2 \in \R $. 
\ifthenelse{\boolean{extendedversion}}{
In addition, under Assumption (A) we have for $ \eta > 0 $ 
\begin{align}
\label{FProp1}
  | f_{0,j}^{(n)}( \vecv_n, \vecw_n ) | &\le C_1 \| \vecv_n \|_{\ell_1} \| \vecw_n \|_{\ell_1} (1 \vee j)^{-3/2-\theta/2}, \qquad j \ge 0, \\
\label{FProp2}
  | f_{l,j}^{(n)}( \vecv_n, \vecw_n ) | &\le C_2 \| \vecv_n \|_{\ell_1} \| \vecw_n \|_{\ell_1} (1\vee j)^{-1-\theta(1/2+\eta)} (1 \vee l)^{-1/2-\theta(1/2-\eta)}, \qquad l, j \ge 1, \\
\label{FProp3}
  | \wt{f}_{l,i}^{(n)}( \vecv_n, \vecw_n ) | &\le C_3 \| \vecv_n \|_{\ell_1} \| \vecw_n \|_{\ell_1} i^{-\theta(1/2+\eta)} (1 \vee l)^{-1/2-\theta(1/2-\eta)}, \qquad l \ge 0, i \ge 1, \\
\label{FProp4}
  | \wt{f}_{l,0}^{(n)} | & \le C_4 \| \vecv_n \|_{\ell_1} \| \vecw_n \|_{\ell_1} (1 \vee l)^{-3/4 - \theta/2}, \qquad l \ge 0,
\end{align}
for constants $C_1, C_2, C_3 $ not depending on $ n $ and the weighting vectors.
}{}

\begin{lemma}
\label{TechnicalLemma}
Suppose that $ \vecv_n, \vecw_n $ have uniformly bounded $ \ell_1 $-norm in the sense of equation~(\ref{l1Condition}). Then Assumption (A) implies 
\begin{equation}
\label{Ass1}
  \sup_{n \in \N} \sum_{i=1}^\infty \sum_{l=0}^\infty (\wt{f}_{l,i}^{(n)} -\wt{f}_{l,i+n'}^{(n)}  )^2 \le C (n')^{1-\theta}, \qquad \text{for all $ n' = 1, 2, \dots $},
\end{equation}
\begin{equation}
\label{Ass2}
  \sup_{n \in \N} \sum_{k=1}^{n'} \sum_{r=0}^\infty (\wt{f}_{r+k,0}^{(n)} )^2 \le C (n')^{1-\theta}, \qquad \text{for all $ n' = 1, 2, \dots $},
\end{equation}
\begin{equation}
\label{Ass3}
  \sup_{n \in \N} \sum_{k=1}^{n'} \sum_{l=0}^\infty (\wt{f}_{l,k}^{(n)} )^2  \le C (n')^{1-\theta}, \qquad \text{for all $ n' = 1, 2, \dots $},
\end{equation}
where the constant may differ from line to line and depends on the weighting vectors only through their $ \ell_1 $--norms.
There exist
\begin{equation}
\label{AlphaN}
 \alpha_n^2 = \alpha_n^2( \vecv_n, \vecw_n ) \ge 0, \qquad n \ge 1, 
\end{equation}
such that
\begin{equation}
\label{Ass4}
  (\wt{f}^{(n)}_{00})^2 \sum_{j=1}^{n'} ( \gamma_{m'+j} - \sigma_{m'+j}^4 )
+ \sum_{j=1}^{n'} \sum_{l=1}^{j-1} ( \wt{f}_{j-l,0}^{(n)} )^2 \sigma_{m'+j}^2 \sigma_{m'+l}^2 - n' \alpha_n^2
\le C (n')^{1-\theta},
\end{equation}
for all $ n', m' = 0, 1, \cdots $. 

Further, if
$ \vecv_n, \vecw_n, \wt{\vecv}_n, \wt{\vecw}_n $, $ n \ge 1 $, have uniformly
bounded $ \ell_1 $-norms, then there exist
\begin{equation}
\label{BetaNuMu}
 \beta_n = \beta_n( \vecv_n,  \vecw_n, \wt{\vecv}_n, \wt{\vecw}_n  ),
 \qquad n \ge 1,
\end{equation}
with 
\begin{align}
\label{Ass5}
& \nonumber \wt{f}_{0,0}^{(n)}( \vecv_n, \vecw_n )
  \wt{f}_{0,0}^{(n)}( \wt{\vecv}_n, \wt{\vecw}_n ) 
  \sum_{j=1}^{n'} ( \gamma_{m'+j} - \sigma^4_{m'+j} ) \\
& \qquad \qquad   +
  \sum_{j=1}^{n'} \sum_{l=1}^{j-1} \wt{f}_{j-l,0}^{(n)}( \vecv_n, \vecw_n ) 
\wt{f}_{j-l,0}^{(n)}( \wt{\vecv}_n, \wt{\vecw}_n ) \sigma^2_{m'+j} \sigma^2_{m'+l} \\ \nonumber
& \qquad \qquad - n' \beta_n(  \vecv_n,  \vecw_n, \wt{\vecv}_n, \wt{\vecw}_n  ) 
\stackrel{n',m'}{<<} (n')^{1-\theta}.
\end{align}
\end{lemma}

\begin{remark} 
\label{TheRemark}
If the moments up to the order $ 4 $ are stationary
such that $ \gamma_k = \gamma $ and $ \sigma_k^2 = \sigma^2 $, say, (\ref{Ass4})
is a consequence of (\ref{Ass2}) and
\[
  \alpha_n^2(\vecv_n, \vecw_n) = (\gamma - \sigma^4) [ \wt{f}^{(n)}_{0,0}(\vecv_n, \vecw_n) ]^2 + \sigma^4 \sum_{l=1}^\infty [\wt{f}_{l,0}^{(n)}(\vecv_n, \vecw_n) ]^2,
\]
as well as
\[
  \beta_n( \vecv_n,  \vecw_n, \wt{\vecv}_n, \wt{\vecw}_n  )
  = \wt{f}^{(n)}_{0,0}(\vecv_n, \vecw_n) \wt{f}^{(n)}_{0,0}( \wt{\vecv}_n, \wt{\vecw}_n )
  (\gamma - \sigma^4) + \sigma^4 \sum_{l=1}^\infty 
  \wt{f}^{(n)}_{l,0}(\vecv_n, \vecw_n) \wt{f}^{(n)}_{l,0}( \wt{\vecv}_n, \wt{\vecw}_n ).
\]
Note also that, in general, (\ref{Ass1})--(\ref{Ass3}) and (\ref{Ass4}) ensure that the quantities $ \alpha_n^2(\vecv_n, \vecw_n) $ and $ \beta_n(\vecv_n,  \vecw_n, \wt{\vecv}_n, \wt{\vecw}_n  ) $ are bounded. 

Suppose that the coefficients of the time series as well as the weighting vectors do not depend on $n$, i.e. $ c_{nj}^{(\nu)} = c_j^{(\nu)} $, $ \nu =1, 2, \dots, d_n $, $ j \ge 0 $, and $ w_{n\nu} = w_\nu $, $ v_{n\nu} = v_\nu $, $ \nu = 1, 2, \dots$. Since then $ \wt{f}_{0,0}^{(n)} = \sum_{j=0}^\infty \left( \sum_{\nu=1}^{d_n} [ c_j^{(\nu)} ]^2 \right)^2 $ and
\[
  \sum_{\ell=1}^\infty [ \wt{f}_{\ell,0}^{(n)}( \vecv_n, \vecw_n ) ]^2
  = \sum_{\ell=1}^\infty \left( \sum_{j=0}^\infty \sum_{\nu, \mu=1}^{d_n} v_\nu w_\mu 
  [ c_j^{(\nu)} c_{j + \ell}^{(\nu)} + c_j^{(\mu)} c_{j+\ell}^{(\nu)} ] \right)^2,
\]
we have $ \alpha_n(\vecv_n, \vecw_n) \to \alpha^* $, as $ n \to \infty $. The limit $ \alpha^* $ is obtained by replacing formally $ d_n $ by $ \infty $. Similarly, one obtains the convergence $ \beta_n( \vecv_n,  \vecw_n, \wt{\vecv}_n, \wt{\vecw}_n  ) \to \beta^* $, as $ n \to \infty $, and an explicit formula for $ \beta^* $.
\end{remark}

\section{Asymptotics}
\label{Asymptotics}

For many sequences, $ X_n $, $ n \ge 1 $, of random variables, such as sums of i.i.d.  random variables with finite second moment, a strong approximation with rate holds true, i.e. after redefining the process on a new probability space, there exists a Brownian motion $ B(t) $, $ t \ge 0 $, such that, on that new space, 
\[
  | X_{t} - \sigma B(t) | = O( t^{1/2-\lambda} ), \qquad \text{for all $t > 0$},
\]
a.s., as $ n \to \infty $, for some positive constant $ \sigma $ and a constant $ \lambda > 0 $.
Results of this type date back to the seminal work of \cite{KMT1975}, \cite{KMT1976} and have been extended and refined since then, in particular to dependent sequences attaining values in a Hilbert space and martingales, see e.g. \cite{Philipp1986}.   Such a strong approximation result also yields an approximation of the rescaled \cadlag process $ n^{-1/2} X_{\trunc{tn}} $, $ t \in [0,1] $, by the Brownian motion $  \sigma B(t) $, $ t \in [0,1] $, and implies the FCLT, i.e. the weak convergence
\[
  n^{-1/2} X_{\trunc{tn}} \Rightarrow \sigma B(t),
\]
as $ n \to \infty $, where $ \Rightarrow $ signifies weak convergence in the Skorohod space $ D[0,1] $. This, in turn, implies the weak convergence of continuous mappings of $  n^{-1/2} X_n( \trunc{n \bullet } ) $. Observe that when $ X_n = \sum_{i=1}^n \xi_i $ is a sum of i.i.d. random variables with $ E(\xi_1^2) < \infty $, then we obtain the classical Donsker theorem and, for $ t = 1 $, the CLT.

In order to obtain strong approximations for the bilinear form $ Q( \vecv_n, \vecw_n ) $, we shall derive a martingale approximation for a partial sum associated to $ Q(\vecv_n, \vecw_n) $, which eases the verification of general conditions for sequences taking values in Hilbert spaces due to \cite{Philipp1986} to obtain a strong approximation. 

We need to introduce further notation. Define
\begin{align}
\label{SigmaDefHat}
  \wh{\bfSigma}_{nk} & = \left( \sum_{i=1}^k Y_i^{(\nu)} Y_i^{(\mu)} \right)_{1 \le \nu, \mu \le d_n}, \\
\label{SigmaDef}
  \bfSigma_{nk} & = \left( \sum_{i=1}^k E Y_i^{(\nu)} Y_i^{(\mu)} \right)_{1 \le \nu, \mu \le d_n},
\end{align}
for $ n, k \ge 1 $. To be precise, our results shall deal with 
\begin{equation}
\label{DefDnk}
  D_{nk} = \vecv_n'( \wh{\bfSigma}_{nk} - \bfSigma_{nk} ) \vecw_n,
  \qquad n, k \ge 1,
\end{equation}
and the associated \cadlag processes
\begin{equation}
\label{DefCalD}
  \calD_n(t) = \vecv_n' n^{-1/2}( \wh{\bfSigma}_{n, \trunc{nt}} - \bfSigma_{n, \trunc{nt}} ) \vecw_n, \qquad t \in [0,1], n \ge 1,
\end{equation}
and
\begin{equation}
\label{DefCalD0}
  \calD_n^0(t) = \calD_n( \trunc{nt}/ n) 
  - \trunc{nt}/n\calD_n(1), \qquad t \in [0,1], n \ge 1.
\end{equation}
Observe that $ \calD_n^0(t) $ does not depend on the true variance-covariance matrices $ \{ \bfSigma_{nk} : 1 \le k \le n \} $. Processes
of this form are therefore frequently used in change-point analysis, 
see, e.g., \cite{CsoergoeHorvath1997} and \cite{Steland2015}.
If the dependence of the above quantities on $ \vecv_n, \vecw_n $ matters, we shall indicate this in our notation and then write
\[
  D_{nk}( \vecv_n, \vecw_n), \calD_n(t; \vecv_n, \vecw_n ), \calD_n^0(t; \vecv_n, \vecw_n). 
\]

Recalling that $ \wh{\bfSigma}_n = n^{-1} \wh{\bfSigma}_{n,n} $, cf. (\ref{DefSigmaHat}) and (\ref{SigmaDefHat}),
\[
  \calD_n(1) = \vecv_n' \sqrt{n}( \wh{\bfSigma}_n - \bfSigma_n ) \vecw_n, \qquad n \ge 1,
\]
is the centered and scaled version of the bilinear form $ Q( \vecv_n, \vecw_n ) $,
where 
\[
  \bfSigma_n = E \wh{\bfSigma}_n = \frac{1}{n} \sum_{i=1}^n E( \vecY_{ni} \vecY_{ni} )'.
\]
If $ \{ \vecY_{ni} : 1 \le i \le n \} $ is stationary, then $ \bfSigma_n $ simplifies to  $ \bfSigma_n = E( \vecY_{n1} \vecY_{n1}') $.

For two weighting vectors $ \vecv_n $ and $ \vecw_n $ we may associate the
martingales
\begin{equation}
\label{ApproxMartingale}
  M_k^{(n)}( \vecv_n, \vecw_n )
  = \wt{f}_{00}^{(n)}( \vecv_n, \vecw_n ) \sum_{i=0}^k (\epsilon_i^2 - \sigma_i^2)
   +
   \sum_{i=0}^k \epsilon_i \sum_{l=1}^\infty \wt{f}_{l,0}^{(n)}( \vecv_n, \vecw_n )
   \epsilon_{i-l}, \qquad k, n \ge 0,
\end{equation}
and the corresponding \cadlag processes
\[
  \calM_n(t; \vecv_n, \vecw_n) = n^{-1/2} M_{\trunc{nt}}^{(n)} ( \vecv_n, \vecw_n ), \qquad t \in [0,1], n \ge 1.
\]
It turns out that, under the assumptions of the article, those martingales are close to $ D_{nk}( \vecv_n,\vecw_n)  $ and $ \calD_n(t;\vecv_n,\vecw_n) $, respectively, which is the key to obtain large-sample asymptotics in terms of strong approximations to infer second order information, i.e. variances and covariances, of projections. 

The following theorem shows that bilinear forms of uniformly bounded $ \ell_1 $-projections satisfy a strong approximation result that implies the functional central limit theorem and therefore also the central limit theorem. Recall that to a standard Brownian motion $ B_n(t) $, $ t \in [0,\infty) $, we may define the rescaled version $ \overline{B}_n(s) = n^{-1/2} B_n(s n ) $, $ s \in [0,1] $, which satisfies $ E( \overline{B}_n(s) \overline{B}_n(t) ) = \min(s,t) $ for $ s, t \in [0,1] $. For convenience, we shall denote $ \overline{B}_n $ also by $ B_n $ and call it {\em the $ [0,1]$--version of $ B_n $}.

\begin{theorem}
\label{Th1} Suppose $ \vecY_{ni} $,  $1 \le i \le n $, $ n \ge 1 $, is a vector time series according to model (\ref{SubordModel}) that satisfies Assumption (A). Let $ \vecv_n $ and $ \vecw_n $ be weighting vectors with uniformly bounded $ \ell_1 $--norm in the sense of (\ref{l1Condition}). Then, for each $n \in \N $, there exists an equivalent version of $ D_{nk}( \vecv_n, \vecw_n) $ and thus of $ D_n(t;\vecv_n, \vecw_n) $, $ t \ge 0 $, again denoted by $ D_{nk}(\vecv_n, \vecw_n) $ and $ D_n(t; \vecv_n, \vecw_n) $, and a standard Brownian motion $ \{ B_n(t) : t \ge 0 \} $, which depends on $ (\vecv_n, \vecw_n) $, i.e. $ B_n(t) = B_n(t;\vecv_n, \vecw_n) $, both defined on some probability space $ (\Omega_n, \mathcal{F}_n, P_n ) $, such that for some $ \lambda > 0 $ and a constant $ C_n $
\begin{equation}
\label{StrongApprox0}
  | D_{nt}(\vecv_n, \vecw_n ) - \alpha_n(\vecv_n, \vecw_n)  B_n(t) | \le C_n t^{1/2-\lambda}, \qquad \text{for all $ t > 0 $ a.s.,}
\end{equation}
and
\begin{equation}
\label{StrongApprox0b}
  \max_{k \le n} \left| D_{nk}(\vecv_n, \vecw_n ) - \frac{k}{n} D_{nn}(\vecv_n, \vecw_n ) - \alpha_n(\vecv_n, \vecw_n )
  [ B_n(k) - \frac{k}{n} B_n(n) ] \right| \le 2 C_n k^{1/2-\lambda}
\end{equation}
for $ k \le n $. If 
\begin{equation}
\label{CN_Cond}
 C_n n^{-\lambda} = o(1),
 \end{equation} 
as $ n \to \infty $, this implies the strong approximation
\begin{equation}
\label{StrongApprox}
  \sup_{t \in [0,1]} | \calD_n(t; \vecv_n, \vecw_n) - \alpha_n(\vecv_n, \vecw_n) B_n(\trunc{nt}/n) | =o(1),
  \qquad \text{a.s.},
\end{equation}
as $ n \to \infty $, for the $ [0,1]$--version of $ B_n $. Further,
\begin{equation}
\label{StrongApprox0_1}
\max_{k \le n} \left| \calD_n^0\left( k/n; \vecv_n, \vecw_n \right) - \alpha_n(\vecv_n, \vecw_n) B_n^0\left( k/n \right) \right| = o(1),
\end{equation}
and
\begin{equation}
\label{StrongApprox0_2}
\sup_{t \in [0,1]} \left| \calD_n^0(t; \vecv_n, \vecw_n) - \alpha_n( \vecv_n, \vecw_n) B_n^0 \left(  \trunc{nt}/ n  \right) \right| = o(1), 
\end{equation}
a.s., as $ n \to \infty $, where $ B_n^0(t) = B_n(t) - t B_n(1) $, $ t \in [0,1] $, denotes the Brownian bridge associated to the $ [0,1] $--version of $ B_n $.
\end{theorem}

Clearly, (\ref{StrongApprox}) implies 
\begin{equation}
\label{MainCLT}
  | \calD_n(1;\vecv_n, \vecw_n) - \alpha_n(\vecv_n, \vecw_n) B_n(1) | = o(1), \qquad \text{a.s.},
\end{equation}
as $ n \to \infty $, i.e. $ \calD_n(1;\vecv_n, \vecw_n) $ is asymptotically $ \calN(0, \alpha_n^2(\vecv_n, \vecw_n) ) $, suggesting the standardized statistic
\[
  \calD_n^*(\vecv_n, \vecw_n) = \alpha_n^{-1}(\vecv_n, \vecw_n) \calD_n(1; \vecv_n, \vecw_n ),
\]
which is asymptotically standard normal by (\ref{MainCLT}),
to draw statistical inference on the covariance of two  projections $  \vecv_n' \vecY_n $ and $ \vecw_n' \vecY_n $. By virtue of the scaling property (\ref{FTildeScaling}), which carries over to $ \alpha_n(\vecv_n, \vecw_n) $, $ \calD_n^*(\vecv_n, \vecw_n) $ depends on the weighting vectors $ \vecv_n $ and $ \vecw_n $ only through the associated unit vectors $ \vecv_n^* = \vecv_n / \| \vecv_n \| $ and $ \vecw_n^* = \vecw_n / \| \vecw_n \| $.

For $ \vecv_n = \vecw_n $ we obtain the asymptotic normality of the statistic
\[
  \alpha_n^{-1}(\vecv_n,\vecv_n)\sqrt{n} [ \wh{\Var}( \vecv_n' \vecY_n ) - \vecv_n' \bfSigma_n \vecv_n ]
\]
which allows to draw inference on the variance of the projection $ \pi_n = \vecv_n' \vecY_n  $ of $ \vecY_n $ onto $ \operatorname{span} \{ \vecv_n \} $. Estimation of the asymptotic variance parameter is discussed below. 

Clearly, it is of interest to project onto, say $K$, vectors
$ \vecv_{n1}, \dots, \vecv_{nK} $, as discussed in greater detail in Subsection~\ref{Subsec: sparsePCA}. The following generalization to a finite number of bilinear forms provides the required multivariate extension.

Recall that $ \vecB(t) $, $ t \ge 0 $, is a Brownian motion in $ \R^K $ with covariance matrix $ \matV > 0 $, 
if $ \vecB(t) = \matV^{1/2} \vecB_0(t) $ for a standard Brownian motion, i.e. $ \vecB_0(s) \sim \calN( \vecnull, s \matid_K )$ and $ E[ (\vecB_0(s)-\vecB_0(t))(\vecB_0(s)-\vecB_0(t))'] = (s-t) \matid_K $ for $ 0 \le s \le t $. This means, $ \vecB(t) $ is a Gaussian process with independent increments, $ E( \vecB(s) \vecB(t)' ) = \min(s,t) \matV $, $ s, t \ge 0 $, and thus $ \vecB(t) $, $ t \ge 0 $, is characterized by $ \vecV = E( \vecB(1) \vecB(1)' ) $. The latter still makes sense if $ \vecV $ is singular.

\begin{theorem} 
\label{ThK}
Let $ \{ \vecv_{nj}, \vecw_{nj}  : 1 \le j \le K \} $ be weighting vectors of dimension $ d_n $ satisfying condition (\ref{l1Condition}). Then, under the assumptions of Theorem~\ref{Th1}, there exists a $K$--dimensional Brownian motion $ \{ \vecB^{(n)}(t) :  t \in [0,1] \} $ with coordinates $ B_{ni} = B_n(t; \vecv_{ni}, \vecw_{ni} ) $, $ t \in [0,1] $, $ i = 1, \dots, K $, characterized by
\[
  E(  B_n(1; \vecv_{ni}, \vecw_{ni} )  B_n(1; \vecv_{nj}, \vecw_{nj} ) ) = \beta_n( \vecv_{ni}, \vecw_{ni}, \vecv_{nj}, \vecw_{nj} ), 
\]
for $ 1 \le i, j \le K $ with $ i \not= j $ and
$
  E(  B_n^2(1; \vecv_{ni}, \vecw_{ni} ) ) = \alpha_n(  \vecv_{ni}, \vecw_{ni} ),
$
for $ i = 1, \dots, K $, such that 
\begin{equation}
\label{StrongApprox2}
  \left\| \left( \calD_n(t; \vecv_{ni}, \vecw_{ni} ) \right)_{i=1}^K   
   - \left(  B_n( \trunc{nt}/n ; \vecv_{ni}, \vecw_{ni}  ) \right)_{i=1}^K \right\| 
   = o(1),
\end{equation}
a.s., as $ n \to \infty $, where $ \| \bullet \| $ denotes an arbitrary vector norm on $ \R^K $.
\end{theorem}

Having in mind the application of the above result to change-point detection, see Subsection~\ref{Subsec: change-points},
the following corollary dealing with the frequently used maximally selected CUSUM statistic is of interest.

\begin{corollary} 
\label{Corollary1}
Suppose that $ \vecY_{n1}, \dots, \vecY_{nn} $ is a  $ d_n$--dimensional vector time series following model (\ref{SubordModel}) which  satisfies Assumption (A). Let $ \vecv_n $ and $ \vecw_n $ be weighting vectors with uniformly bounded $ \ell_1$--norm in the sense of (\ref{l1Condition}). Then, after redefining the series on a new probability space, there exist  standard Brownian motions $ B_n $ on  $[0,1] $, such that
\begin{equation}
\label{StrongApprox3}
  \left|
    \max_{k \le n} | \calD_n(k/n;\vecv_n, \vecw_n) | - \alpha_n( \vecv_n, \vecw_n ) \max_{k \le n} |  B_n( k/n ) | 
  \right|   = o(1),
\end{equation}
and
\begin{equation}
\label{StrongApprox0_3}
  \left|
    \max_{k \le n} | \calD_n^0(k/n;  \vecv_n, \vecw_n ) |
    - 
    \alpha_n( \vecv_n, \vecw_n )
    \max_{k \le n} | B_n^0( k/n ) | 
  \right|
    = o(1),
\end{equation}
for the Brownian bridge $ B_n^0 $ associated to $ B_n $, a.s., as $ n \to \infty $.
\end{corollary}

We conjecture that the strong approximation (\ref{StrongApprox0}) holds with the rate $ O(n^{-\lambda}) $, i.e. one can select a universal constant $C$ in (\ref{StrongApprox0}), such that in turn the strong approximations (\ref{StrongApprox})--(\ref{StrongApprox0_3})
would hold with the rate $ O(n^{-\lambda}) $ as well. This is also plausible from the following more general result which provides us with a strong approximation with that rate when weighting over the sample sizes with $ \ell_1 $-bounded weights: Attach to each sample size $n$ a  weight $ \lambda_n $ such that  $ \sum_{n=1}^\infty |\lambda_n| < \infty $. Define
\begin{equation}
\label{GeneralStat}
  D_k( \{ \vecv_n, \vecw_n \} ) =
  \sum_{n,m=1}^\infty \lambda_n \lambda_m \vecv_n' ( \wh{\bfSigma}_{nmk} - \bfSigma_{nmk} ) \vecw_m,
  \qquad k \ge 1,
\end{equation}
where 
\begin{equation}
\label{DefSigmaGeneral}
 \wh{\bfSigma}_{nmk} = \sum_{i \le k} \vecY_{ni} \vecY_{mi}' \quad \text{and} \quad  \bfSigma_{mnk} = E( \wh{\bfSigma}_{nmk} ) 
\end{equation} 
for $ n, m \ge 1 $, based on the full set of observations $ \{ Y_{i}^{(\nu)} : 1 \le \nu < d_n, 1 \le i \le n, n \ge 1 \} $. At this point it is worth recalling that  in model (\ref{SubordModel}) the $d_n$ coordinate processes of $ Y_{ni} $ are well defined for all $ i \ge 1 $. For fixed $k$, $ D_k ( \{ \vecv_n, \vecw_n \} ) $ depends on the array $ \{ Y_{ni}^{(\nu)}, i \le k, 1 \le \nu \le d_n,  n \in \N \} $. Notice that $ D_{nk}( \vecv_n, \vecw_n) $, defined in (\ref{DefDnk}), appears as a special case of (\ref{GeneralStat}). We may define the following \cadlag process associated to (\ref{GeneralStat}) 
\begin{equation}
\label{GeneralCadlag}
  \calD_N(t; \{ \vecv_n, \vecw_n \} ) = N^{-1/2} D_{\trunc{Nt}}( \{ \vecv_n, \vecw_n \} ),
  \qquad t \in [0,1], 
\end{equation}
for some sample size $N$ (which can be equal to $n$).

We have the following general large sample approximation result.

\begin{theorem}
\label{Th2} Let $ \{ \vecv_n \} $ and  $ \{ \vecw_n \} $ be two sequences of weighting vectors with uniformly bounded $ \ell_1 $--norm in the sense of (\ref{l1Condition}). Assume that $ \{ Y_{i}^{(\nu)} : 1 \le \nu < \infty, 1 \le i \le n, n \ge 1 \} $ follows model (\ref{SubordModel}) and satisfies Assumption (A). Let $ \{ \lambda_n \} $ be weights with $ \sum_{n=1}^\infty | \lambda_n | < \infty $. Then there exist  constants $ \alpha( \{ \vecv_n, \vecw_n \} ) $, $ \lambda > 0$ and $ C $ (not depending on the sample size), such that for equivalent versions and a standard Brownian motion $B$ on  $[0,\infty)$, defined on a new probability space,
\begin{equation}
\label{GeneralIP1}
  \left| D_t( \{ \vecv_n, \vecw_n \} ) - \alpha( \{ \vecv_n, \vecw_n \} ) B(t) \right| \le C t^{1/2-\lambda}, 
\end{equation}
a.s., for all $ t > 0 $, leading to the strong invariance principle with rate,
\begin{equation}
\label{GeneralIP2}
  \sup_{t \in [0,1]} | \calD_N(t, \{ \vecv_n, \vecw_n \} )  - \alpha( \{ \vecv_n, \vecw_n \} )  B( \trunc{N t} / N ) | \le C N^{-\lambda},
\end{equation}
a.s., for the $[0,1]$--version of $ B $, where $C$ is the same constant as in (\ref{GeneralIP1}) and may depend on the sequence $ \{ \lambda_n \} $.
\end{theorem}

It remains to discuss how one may estimate the asymptotic variance parameters $ \alpha^2 = \alpha^2( \vecv, \vecw ) $ and
$ \beta^2 = \beta^2( \vecv_r, \vecw_r, \vecv_s, \vecw_s ) $ arising in the strong approximations. 
For a sequence of lag truncation constants $m = m_n $, $n \ge 1 $, and weights $ \{ w_{mh}  : h \in \Z, m \in \N \} $ define the
estimator
\begin{equation} 
\label{LRVEst} 
  \wh{\alpha}_n^2 = \wh{\alpha}_n^2(d) = \wh{\Gamma}_n(0) + 2 \sum_{h=1}^{m} w_{mh} \wh{\Gamma}_n( h ),
\end{equation} 
where
\[ 
   \wh{\Gamma}_n( h ) 
   = \frac{1}{n} \sum_{i=1}^{n-h} 
   [ Y_i^{(v)} Y_i^{(w)} - \wh{\mu}_n^{(v,w)} ]
   [ Y_{i+|h|}^{(v)} Y_{i+|h|}^{(w)} - \wh{\mu}_n^{(v,w)} ],
   \qquad |h|<n,
\]
with $ \wh{\mu}_n^{(v,w)} = n^{-1} \sum_{j=1}^n Y_j^{(v)} Y_j^{(w)} $ 
and $ Y_i^{(v)} = \vecv' \vecY_i = \sum_{j=0}^\infty c_j^{(v)} \epsilon_{i-j} $, $ i \ge 1 $, with weights $ c_j^{(v)} = \left( \sum_{\nu=1}^d c_j^{(\nu)} v_j \right) $, $ j \ge 0 $. At this point
we consider the dimension as a further parameter, such that $
Y_i^{(v)} $, $ \wh{\Gamma}_h = \wh{\Gamma}_h(d) $ and $ \alpha^2 = \alpha^2(d) $ depend on the dimension $ d $. 

More generally, the covariance parameters $ \beta_{rs} = \beta( \vecv_r, \vecw_r; \vecv_s, \vecw_s ) $, $ 1 \le r, s \le K $, arising in Theorem~\ref{ThK} can be estimated by
\begin{equation} 
\label{LRVEst} 
  \wh{\beta}^2(r,s) 
  = \wh{\beta}^2(r,s;d) 
  = \wh{\Gamma}_{n}^{(r,s)}(0) + 2 \sum_{h=1}^{m} w_{mh} \wh{\Gamma}_n^{(r,s)}( h ),
\end{equation} 
where $ m = m_n $ is a lag truncation sequence and
\[
  \wh{\Gamma}_n^{(r,s)}( h ) = \frac1n 
    \sum_{k=1}^{n-h} 
    [ Y_k( \vecv_r ) Y_k( \vecw_r ) - \wh{\mu}_n(r) ]
    [ Y_{k+|h|}( \vecv_s ) Y_{k+|h|}( \vecw_s ) - \wh{\mu}_n(s) ],
    \qquad |h| < n,
\]
with $ \wh{\mu}_n(r) = n^{-1} \sum_{j=1}^n Y_j( \vecv_r ) Y_j( \vecw_r ) $, for $ 1 \le r, s \le K $.

For the weights we impose the following standard assumptions.
\begin{itemize}
  \item[(W1)] $ w_{mh} \to 1 $, as $ m \to \infty $, for all $ h \in \Z $.
  \item[(W2)] $ 0 \le w_{mh} \le W < \infty $ for some constant $W$, for all $ m \ge 1 $, $ h \in \Z $.
\end{itemize}
Typically, the weights are defined by a kernel function $w$ via
$ w_{mh} = w(h/b_m) $ for some bandwidth parameter $ b = b_m $. Examples are the truncated kernel, $ k_{tr}(x) = \eins( |x| \le 1) $ with $ b = m+1 $, or Bartlett's estimator given by the triangular weight function $ w(x) = (1 - x) \eins( |x| \le 1 ) $, see \cite{NeweyWest1987} and \cite{Andrews1991} amongst others. 

\cite{Jirak2012} has studied estimation of long run variance parameters associated to $ d \to \infty $ nonlinear time series, based on results of \cite{Wu2009}. Those results are not applicable to our general setting,
since the coefficients of the linear processes $ Y_i^{(v)} $ depend on $n$ if $ d = d_n $.
The following theorem provides the $ L_1 $--consistency, uniformly
over $d$ and a possibly infinite number of weighting vectors, as long as they are taken from an $ \ell_1 $--bounded set. This yields the consistency of $ \wh{\alpha}_n^2(d_n) $ for a growing dimension $ d_n \to \infty $, since
\[
  E | \wh{\alpha}_n^2(d_n) - \alpha^2(d_n) | 
  \le \sup_{d \in \N} E | \wh{\alpha}^2_n(d) - \alpha^2(d) |,
\]
without a constraint on the growth of $ d_n $, thus going beyond the known results.

\begin{theorem}
\label{ThLRVEst}
Assume (W1) and (W2) and suppose that $ c_{nj}^{(\nu)} = c_j^{(\nu)} $ for $ \nu = 1, 2, \dots $, $ j \ge  0,  n \ge 1, $ satisfy the decay condition
\begin{equation}
  \label{DecayCoeff}
    \sup_{1 \le \nu} | c_j^{(\nu)} | << (j \vee 1)^{-(1+\delta)}
\end{equation}
for some $ \delta > 0 $.  Suppose that $ \epsilon_k $ are i.i.d. with $ \max_k E | \epsilon_k |^{8} < \infty $ and assume that $ m = m_n \to \infty $ with $ m^2/n = o(1) $, as $ n \to \infty $. If $ \vecv, \vecw $ are weighting vectors with $ \| \vecv \|_{\ell_1}, \| \vecw \|_{\ell_1} < \infty $, then
\begin{equation}
 \label{Consistency1}
   \sup_{d \in \N} E | \wh{\alpha}_n^2(d) - \alpha^2(d) | \to 0.
 \end{equation}
 Further, if  $ \vecv_\ell, \vecw_\ell $, $ \ell \ge 1 $, are weighting vectors
 with uniformly bounded $ \ell_1 $--norm, i.e.
  \[
  \sup_{1 \le \ell} \max \{ \| \vecv_\ell \|_{\ell_1},  \| \vecw_\ell \|_{\ell_1} \} \le C_{v,w} < \infty,
\]
for some constant $ C_{v,w} $, then
 \begin{equation}
 \label{Consistency2}
    \sup_{1 \le r, s} \sup_{d \in \N} E | \wh{\beta}_n^2(r,s;d) - \beta^2(r,s;d) | \to 0,
 \end{equation}
 as $ n \to \infty $.
\end{theorem}

Suppose that the dimension is a random variable, $ D $, and drawn according to some (prior) probability measure $ Q $ on $ \N $, such that $D$ and $ \{ \vecY_{ni} \} $ are independent, and statistical inference is conducted given $ D = d$. A natural measure to evaluate the estimator $ \wh{\alpha}_n^2(D) $ in this setting is the expected mean deviation,
\[
  EMD(\wh{\alpha}_n^2(D)) = \int E | \wh{\alpha}_n^2(d) - \alpha^2(d) | \, dQ(d).
\]
Theorem~\ref{ThLRVEst} readily implies 
  $
    EMD( \wh{\alpha}_n^2(D) ) \to 0,
 $  as $ n \to \infty $.

\section{Applications}
\label{Applications}

The results of the present paper have direct applications to several problems and procedures, respectively, which are extensively studied for high--dimensional time series, especially for big data. They contribute novel large-sample approximations for making inference based on the corresponding statistics, usually projections.

\subsection{Optimal portfolio selection}
\label{Subsec: portfolio}

The problem of optimal portfolio selection, dating back to Markowitz' seminal work, see \cite{Markowitz1952}, is an intrinsically high-dimensional problem. We are given a usually large number $d_n$ of assets and associated returns  $ \vecR_n = (R_n^{(1)}, \dots, R_n^{(d_n)})' $ corresponding to the time period $ [n-1,n] $ with mean vector $ \bfmu $ and covariance matrix $ \bfSigma = ( \sigma_{ij} )_{ij} $. Let us assume that the asset return vector time series satisfies the standing assumptions of the paper. Since $ \sigma_{ij} $ is the covariance between
the return of asset $i$ and asset $j$, $ 1 \le i, j, \le d_n $, it is not restrictive, by the very nature of the problem, to assume that the entries
of $ \bfSigma $ neither depend on $ n $ nor $d_n $. 
Suppose that an investor holds at time $n-1$ the position $ w_{nj} $ in asset $ j $, where $ w_{nj} > 0 $ represents a long position and $ w_{nj} < 0 $ a short position. W.l.o.g. we may assume that the initial value (capital) at time $ n-1 $ equals $ V = \sum_{j=1}^{d_n} w_{nj} = 1 $, such that the value  at time
instant $n$ is $ \vecw_n'\vecR_n $.
A classical formulation of the portfolio optimization problem is to minimize the risk, defined as the variance, associated to the portfolio return $ \vecw_n'\vecR_n $ at time $n$, i.e. to consider the problem
\[
  \min_{\vecw_n} \Var( \vecw_n' \vecR_n ) = \vecw_n' \bfSigma \vecw_n,
  \qquad \text{subject to $ \vecw_n' \eins = 1 $},
\]
whose solution is known to be 
$
  \vecw_n = ( \eins' \Sigma^{-1} \eins )^{-1} \eins' \Sigma^{-1}.
$
Here $ \eins $ is the $n$-vector with unit entries. Obviously, if that optimal solution satisfies the no-short-sales
condition $ \vecw_n \ge 0 $, then
$ 
  \|\vecw_n \|_{\ell_1} = \vecw_n'\eins = ( \eins' \Sigma^{-1} \eins )^{-1} \eins' \Sigma^{-1} \eins = 1,
$
such that the optimal portfolio has uniformly bounded $ \ell_1 $-norm. The mean-variance formulation
adds a constraint on the  target mean portfolio return and thus considers the problem
\[
  \min_{\vecw_n} \vecw_n' \bfSigma \vecw_n, \qquad \text{subject to $ \vecw_n'\eins = 1, \ \vecw_n'\mu = \mu_0 $}
\]
for  some $ \mu_0 $. The solution is
\[
  \vecw_n = \frac{c - \mu_0 b}{ac - b^2} \bfSigma^{-1} \eins +
    \frac{ \mu_0 a - b}{ ac - b^2} \bfSigma^{-1} \bfmu.
\]
with $ a = \eins' \bfSigma^{-1} \eins $, $ b = \eins' \bfSigma^{-1} \bfmu $ and $ c = \bfmu' \bfSigma^{-1} \bfmu $. Based on estimates of $ \bfSigma^{-1} $ and $ \bfmu $ from past data, one calculates the optimal portfolio which is then held until the next rebalancing. If the dimension is large, inverting the sample covariance matrix may result in substantial numerical instability and becomes impossible if the dimension is larger than the sample size. Shrinking is a commonly applied approach for regularization, in order to obtain a stable and invertible estimator, see \cite{LedoitWolf2003}, amongst others. (We also refer to our subsection 5.3 for more on this.) It is worth mentioning that adding a non-negativity constraint (i.e. no-short-sales)  has been observed to have a similar regularizing effect, see \cite{Jagannathan2003}. 

To obtain sparse portfolios,  \cite{BrodieEtAl2009} proposed to add an $ \ell_1$-constraint. To discuss their
approach, let us assume that we are given an additional independent learning sample 
$ ( r_t^{(j)} )_{t,j} $, $ 1 \le j \le d_N $, $ 1 \le t \le N $, of size $N$ of returns
for the same assets, such that in particular $ d_N = d_n $, which is used to
estimate the optimal weighting vector. Markowitz' problem is equivalent to 
\[
  \min_{\vecw_n} E ( | \mu_0 - \vecw_n'\vecR_n |^2 ), \qquad \text{subject to $ \vecw_n'\bfmu = \mu_0, \ \vecw_n'\eins = 1$},
\]
which suggests the following empirical version 
\begin{equation}
\label{BrodieUnregularized}
  \min_{\vecw_n} n^{-1} \| \mu_0 \eins - \calX_N \vecw_n \|_{\ell_2}^2, \qquad \text{subject to 
$  \vecw_n'\wh{\bfmu}_N = \mu_0, \ \vecw_n'\eins = 1$,}
\end{equation}
where $ \calX_{N} = ( r_t^{(j)} )_{t,j} $ is the $N \times d_n$ dimensional data matrix of returns with rows $ \vecr_t' = ( r_t^{(1)}, \dots, r_t^{(d_n)}  ) $, $ t = 1, \dots, N $, and $ \wh{\bfmu}_N = N^{-1} \sum_{t=1}^N \vecr_t $,
see \cite[formula 1]{BrodieEtAl2009}. These authors examine the $\ell_1$-regularized version of (\ref{BrodieUnregularized}), 
\[
  \min_{\vecw_n} n^{-1} \| \mu_0 \eins - \calX_N \vecw_n \|_{\ell_2}^2 + \rho \| \vecw_n \|_{\ell_1}, \quad
  \text{subject to $ \vecw_n'\wh{\bfmu}_n = \mu_0, \ \vecw_n'\eins = 1 $},
\]
for some regularization parameter $ \rho > 0$.
Whereas for large values of $ \rho $ the classical solution is recovered, smaller values lead to effective $ \ell_1 $-penalization and sparse portfolio vectors with only relatively few active positions.

\subsection{Projections onto lower-dimensional subspaces, sparse principal components and the LASSO}
\label{Subsec: sparsePCA}

A primary goal of multivariate statistical analysis and an indispensable tool to investigate big data is to project high--dimensional data onto lower-dimensional subspaces. Therefore the results of this paper directly address various approaches that are based on $ \ell_1 $ projection vectors to ensure sparse representations of the high--dimensional data. In particular, the results apply to the recently proposed methods of \cite{JollifeEtAl2003}, \cite{ShenHuan2008}, \cite{TibshiraniWittenHastie2009} for sparse principal component analysis and the LASSO, see \cite{Tibshirani1996}.

Recall that for $ L $ vectors $ \vecw_n^{(k)} $, $ k = 1, \dots, L $,  of dimension $ d_n $ we may define the linear mapping $ \pi_n : \R^{d_n} \to \R^{d_n} $,
\[
  \pi_n = \matP_n \matP_n' \vecY_n, \qquad \matP_n = [ \vecw_n^{(1)}, \dots, \vecw_n^{(L)} ],
\]
onto the associated linear subspace $ \text{span} \{ \vecw_n^{(k)}  : k = 1, \dots L \} $,
which is the orthogonal projection if the $ \vecw_n^{(k)}  $ are orthonormal.
Otherwise, the latter is given by  $ \matP_n(\matP_n'\matP_n)^- \matP_n' $, where $ \matA^- $ denotes a generalized inverse of a square matrix $ \matA $. 
In both cases, inferential procedures for $ \pi_n $ can be based on the asymptotics of the {\em dimension-reducing statistic} $ \matP_n' \vecY_n $, such that our results apply, provided the projection vectors in use and the vector time series satisfy our assumptions. However,
in general, the columns of $ \matP_n' $ and $ \matP_n(\matP_n'\matP_n)^- \matP_n' $, respectively, do not necessarily satisfy the uniform $ \ell_1 $--condition.

One may apply the following ad hoc approach to transform each vector $ \vecw_n^{(j)} $ such that it satisfies an $ \ell_1 $-constraint 
$
  \| \vecw_n^{(j)} \|_{\ell_1} \le c $, $ j = 1, \dots, L,
$
for some preassigned constant $ c $, by calculating the $ \ell_1 $-constrained optimal projection of $ \vecw_n^{(j)} $ onto $ \text{span}\{ \vecw_n^{(j)} \} $, 
\[
  \max_{\vecu} \vecu' \vecw_n^{(j)} \qquad \text{subject to $ \| \vecu \|_{\ell_2}^2 \le 1, \| \vecu \|_{\ell_1} \le c $},
\]
whose solution is known to be
\[
  \wt{\vecw}_n^{(j)} = S( \vecw_n^{(j)}, \delta ) / \| S( \vecw_n^{(j)}, \delta ) \|_{\ell_2}.
\]
Here
$
  S( a, \delta ) = \text{sgn}(a)(|a|  - \delta)_+
$
is the {\em soft-thresholding function}, $ x_+ = x $ if $ x > 0 $ and $ =0 $ otherwise, and $ \delta  \ge 0 $ is chosen such that $ \| S( \vecw_n^{(j)}, \delta ) \|_{\ell_1} = c $, see  \cite[Lemma~2.2]{TibshiraniWittenHastie2009} and \cite{ShenHuan2008}.

Imposing $ \ell_1 $-constraints is also the basic idea behind most approaches to
define a sparse principal component analysis for high-dimensional data, which  aims at determining
lower-dimensional subspaces generated by $ \ell_1 $-vectors in such a way that they
explain a large part of the variation in the observed data and also provide 
low-rank approximation of the data matrix. 
Let $ \calX_n $ be a $n \times d_n $-dimensional data matrix with centered columns corresponding to an independent sample of size $n$ of the variables $ Y^{(1)}, \dots, Y^{(d_n)} $, whose columns are assumed to be centered. The simplified component technique-lasso (SCoTLASS) approach of \cite{JollifeEtAl2003} defines the first sparse principal component as a solution of the optimization problem
\[
  \max_{\vecv} \vecv' \calX_n'\calX_n \vecv, \qquad \text{subject to $ \| \vecv \|_{\ell_2}^2 \le 1 $, $ \| \vecv \|_{\ell_1} \le c$.}
\]
Further sparse components are obtained by maximizing the same objective function under above constraints and the additional constraints that the further component is orthogonal to the previous components. In this way, after $L$ steps we obtain $L$ orthogonal $\ell_1$-vectors.

In a similar way, \cite{TibshiraniWittenHastie2009} propose a sparse principal component analysis by solving, for the first component, the penalized matrix decomposition problem (PMD) with $ \ell_1$-constraints,
\[
  \max_{\vecu, \vecv} \vecu' \calX_n'\calX \vecv, \qquad \text{subject to $ \| \vecv \|_{\ell_1} \le c $, 
  $\| \vecu \|_{\ell_2}^2 \le 1, \| \vecv \|_{\ell_2}^2 \le 1$},
\]
Observing that for fixed $ \vecv $ the solution is given by $ \vecu = \calX_n  \vecv / \| \calX_n \vecv \|_{\ell_2} $, the first sparse principal component of PMD with $ \ell_1 $-constraints also solves SCoTLASS, see \cite[p.~525]{TibshiraniWittenHastie2009}. However, the PMD approach does not constrain the further components to be orthogonal, so that it differs from SCoTLASS.
Closely related is the sparse PCA (SPCA) method of \cite{ShenHuan2008}. They consider the problem to determine a regularized low-rank matrix approximation,
\[
  \min_{\vecu, \vecv} \| \calX - \vecu \vecv' \|_F^2 + p_\rho(\vecv),
  \qquad \| \vecu \|_{\ell_2} = 1,
\]
for several penalty terms $ p_\rho(\vecv) $ including the case $ \rho \| \vecv \|_{\ell_1} $, for some $ \rho > 0 $.

The LASSO, see \cite{Tibshirani1996} and \cite{Tibshirani2011}, is a well established approach to determine $ \ell_1 $--sparse coefficient vectors in a high--dimensional linear regression model
\[
  Y_t = \vecX_t' \bfbeta_0 + \varepsilon_t,  \qquad E(\varepsilon_t|\vecX_t) = 0, 
  \qquad t = 1, \dots, n,
\]
where the conditional expectation $E(Y_t|\vecX_t) = \vecX_t'\bfbeta_0 $ is the $ L_2 $--optimal predictor for $ Y_t $ given $ \vecX_t $. Given some estimator $ \wh{\bfbeta}_n $ of the unknown coefficient vector $ \bfbeta_0 \in R^d $, the linear projection $ \pi_n(\vecX) = \vecX' \wh{\bfbeta}_n $ is used to predict the outcome of the response for some (future) observed $ \vecX $. The LASSO minimizes the $ \ell_1 $--constrained least squares criterion,
\[
  \bfbeta \mapsto \sum_{t=1}^n (Y_t - \vecX_t'\bfbeta)^2, \qquad \| \bfbeta \|_{\ell_1} \le c,  
\]
for some bound $ c > 0 $ for the $ \ell_1 $--norm, such that the resulting estimator $ \wh{\beta}_n $ is $ \ell_1 $--sparse. Consequently, our results can be applied to draw inference on the variance of the LASSO-based prediction $ \pi_n( \vecX ) = \vecX' \wh{\bfbeta}_n $, provided the regressor vector time series satisfies the assumptions of this article. Among the diverse applications where
the prediction of the response in the presence of a large number of correlated explanatory
variables is of interest, is the analysis of genetic association studies, see \cite{AllisonEtAl2006} for its basic analysis and \cite{BiedermannEtAl2006} for
nonparametric tests and further discussion. In such studies the
regressors are gene expression data and the response is a phenotype. A sparse coefficient vector with only a few nonvanishing entries may allow to identify (groups of) genes which are  associated to the phenotype.

\subsection{Shrinkage estimation}


In many applications, from a statistical point of view, when estimating the common
variance-covariance matrix $ \bfSigma_n $ of a stationary vector time series
$ \vecY_{n1}, \dots, \vecY_{nn} $ of dimension $ d_n $,
one has an interest to regularise $ \wh{\bfSigma}_n $ to improve its (finite-sample) properties such as its mean-squared error $E[\| \wh\bfSigma - \bfSigma_n\|^2_F]$ or its condition number, defined to be the ratio of its largest to its smallest eigenvalue. This is of particular interest if one needs an invertible estimator of  $ \bfSigma_n $. One well-established possibility to  regularise $ \wh{\bfSigma}_n $ (\cite{LedoitWolf2004}, \cite{Sancetta2008}) 
is to consider a shrinkage estimator defined by a linear (in fact a convex) combination of  $ \wh{\bfSigma}_n $ with a well-conditioned "target". Already in the population, for situations where the dimensionality $d_n$ is in the order magnitude of the sample size $n$, shrinkage of the high-dimensional variance-covariance matrix $\bfSigma_n$ towards a target, similarly to ridge regression, can reduce a potentially large condition number. This is achieved by reducing the dispersion of the eigenvalues of $ \bfSigma_n $ around its "grand mean"  $\mu_n := d_n^{-1}\trace{\bfSigma}_n$: large eigenvalues are pulled down towards $\mu_n $, small eigenvalues are lifted up to $\mu_n$. 
Improvement of the mean-squared error $E[\| \wh\bfSigma - \bfSigma_n\|^2_F]$ is achieved via a potentially tremendous variance reduction (due to stabilisation via regularisation), even if obviously a bias is introduced by adding a deliberately misspecified shrinkage target of low complexity (but high regularity) which usually underfits the true underlying variance-covariance matrix. 

A comparatively straightforward, but in practice often well working, choice of the target is a multiple of the $d_n-$dimensional identity matrix $\matid_n $ (\cite{LedoitWolf2004}, \cite{FiecasFrankeSachs2014}). Other choices consist in specifying, e.g. in a context of an economic time series panel, a given or a latent factor which describes the "mean-behaviour" of the panel well in terms of a low-dimensional and hence very stable approximation to the high-dimensional panel structure (\cite{LedoitWolf2003}, \cite{BoehmSachs2008}). Similarly, adding a parametric estimator of small complexity to the fully nonparametric sample estimator as in \cite{FiecasSachs2013}, follows the same  aforementioned paradigm of reducing variance by adding a (model) bias.

The success of this approach, quite naturally, lies in the correct specification of the shrinkage weight $W_n$, the proportion with which the shrinkage target enters into the convex combination: obviously it has to be the higher the less regular the given variance-covariance matrix. More specifically, in the above mentioned literature, a theory of optimal choice of $W_n$ has been delivered, for various scenario, by minimising the mean squared error between the shrunken estimator $ \bfSigma_n^s$ and the true variance-covariance matrix $ \bfSigma_n$. Hence, let
\[
  \bfSigma_n^s =  \bfSigma_n^s(W_n) = (1-W_n) \wh{\bfSigma}_n\ +\ W_n\ \mu_n \matid_n,
\]
which shrinks the sample covariance matrix towards the shrinkage target $ \mu_n \matid_n $.
In the population, the optimal shrinkage weight is derived as
\[
W_n^* = \text{argmin}_{W_n \in [0,1]} d_n^{-1} E[\| \bfSigma_n^s(W_n) - \bfSigma_n\|^2_F] \ ,
\]
leading to the MSE-optimally shrunken matrix $\bfSigma_n^* =  \bfSigma_n^s (W_n^*) $.
A closed form solution can be derived as
\[
W^*_n = E[\| \wh\bfSigma - \bfSigma_n\|^2_F] / E[\| \mu_n \matid_n - \wh\bfSigma_n\|^2_F ]\ , 
\]
where one observes the trade-off between the distance of $\wh\bfSigma_n$ towards $ \bfSigma_n$ (being large in case of a badly conditioned sample covariance matrix) and the distance of the sample estimator to the shrinkage target. 
This choice leads to a true improvement on the level of the mean-squared error:
\[
E[\| \bfSigma_n^* - \bfSigma_n \|^2_F] < E[\| \wh{\bfSigma}_n - \bfSigma_n \|^2_F] \ .
\]

It is obvious that our distributional results of Section 4 can be directly applied to the shrunken matrix  $\bfSigma_n^s$, provided the stationary vector time series 
$ \vecY_{n1}, \dots, \vecY_{nn} $  satisfying (\ref{SubordModel}) and Assumption A, this being the first step in the direction of some inference theory for this kind of shrinkage estimators (which is still lacking in the literature).
In practice, the population quantities $ \mu_n$ and $W^*_n $ need to be replaced by some estimators. In the situation of disposing of independent copies of the sampled data, a possibility is to use those, quite analogously to Section 5.2 and many other "statistical learning situations", in order to construct these estimators $  \wh{\mu}_n $ and $\wh{W}_n$. Then, the discussed results on inference on $\bfSigma_n^*$ continue to hold (conditionally on the "learning sample"). 

\subsection{Change-point analysis}
\label{Subsec: change-points}

Change-point analysis is concerned with the detection and analysis of possible structural changes in the distribution of observations and the determination of the time points of their occurence called change-points. For general methodological overviews we refer to \cite{CsoergoeHorvath1997}, \cite{Steland2012}, the recent review \cite{HuskovaHlavka2012} and the references given in these papers, amongst others. In view of Theorem~\ref{Th1} and Corollary~\ref{Corollary1}, we are in a position to study  an a posteriori (off-line) change-in-variance problem for the variance of a projection $ \vecw_n' \vecY_n $ due to a change in the covariance structure. Off-line procedures are conducted after having observed the full sample and aim at testing for the presence
of a change-point within the sample. In case that such a change-point test rejects the null hypothesis of no change, one is interested in estimating the location of the change-point as well. A test for the change in the covariance matrix has been also proposed by \cite{Jaruskova2013}, but only for independent Gaussian random vectors. For a fixed number of time series see \cite{AueHoermannHorvath2009}.

Suppose that under the null hypothesis of no change $ \vecY_{n1}, \dots, \vecY_{nn} $ forms a $ d_n $--dimensional mean zero stationary vector time series with variance-covariance matrix $ \bfSigma_n^{(0)} $ and satisfying the assumptions of Corollary~\ref{Corollary1}. Our change-point model is formulated in terms of the
sequence of variances of the projections which are determined by the variance-covariance matrices
\[
  \bfSigma_n[i] = \Cov( \vecY_{ni} ), \qquad 1 \le i \le n.
\]
Under the change-point alternative hypothesis we assume that this
sequence is equal to $ \bfSigma_n^{(0)} $ up to the change-point $ q \in \{ 1, \dots, n-1 \} $ and changes for $ i > q $ in such a way that for appropriately selected $ \vecw_n $ satisfying the uniform $ \ell_1 $--condition (\ref{l1Condition}),
\[
  \sigma_n^2(i) = \Var( \vecw_n' \vecY_{ni} ) = \vecw_n' \bfSigma_{n}[i] \vecw_n 
\]
changes from $ \sigma_{n0}^2 = \vecw_n' \bfSigma_n^{(0)} \vecw_n $ to some different value $ \sigma_{n1}^2 \not= \sigma_{n0}^2 $ and then remains constant again. The applications discussed above in Subsections~\ref{Subsec: portfolio} and \ref{Subsec: sparsePCA} provide examples for the selection of the projection vector. However, it can also be chosen in order to analyze certain elements of the variance-covariance matrix. 

An appropriate change-point test statistic directly suggested by our results for the case of known $ \bfSigma_n^{(0)} $ is given by
\[
  V_n = \max_{k \le n} | \alpha_n^{-1} \calD_n(k/n) | = \max_{k \le n}
  n^{-1/2} \alpha_n^{-1}| \vecw_n'( \wh{\bfSigma}_{nk} - k \bfSigma_{n}^{(0)} ) \vecw_n |,
\]
cf. (\ref{SigmaDefHat}) and (\ref{SigmaDef}). Corollary~\ref{Corollary1} now provides us with the asymptotic null distribution,
\[
  \max_{k \le n} | \alpha_n^{-1} \calD_n(k/n) |  \sim_{n \to \infty} \sup_{t \in [0,1]} | B(t) |,
\] 
that is needed to determine critical values in order to devise such a test. Critical values $ c_{1-\alpha} $, $ \alpha \in (0,1) $, can be easily calculated using the well known explicit formula
for the d.f. of  $   \sup_{t \in [0,1]} | B(t) |  $,
see e.g. \cite{Shorack2000}. If $ \alpha_n $ is unknown, it can be estimated by $ \wh{\alpha}_{n_0} $ from a learning sample of size $ n_0 \ge n $, which satisfies the no-change null hypothesis and
the assumptions of Theorem~\ref{ThLRVEst}, using the  first $n$ time series. Assuming that such a learning sample is given is, however, standard in the change-point literature, see e.g. \cite{ChuStinchcombeWhite1996} where it has been named  {\em non-contamination assumption}. We now reject the no--change null hypothesis in favor that a change has occured, if $ V_n > c_{1-\alpha} $. In this case, the unknown  (first) change-point, i.e. the onset, is estimated canonically by
\[
  \wh{k}_n = \min \{ k \le n : | \calD_n( k/n ) | \ge | \calD_n( \ell/n ) |, \ \ell = 1, \dots, n \}.
\]
If $ \bfSigma_n^{(0)} $ is unknown, one may rely on $ \calD_n^0 $ defined in (\ref{DefCalD0}) and use the test statistic 
\[
  V_n^0 = \alpha_n^{-1} \max_{k \le n} | \calD_n^0( k/n ) |
  \stackrel{H_0}{\sim}_{n \to \infty} \sup_{t \in [0,1]} | B^0(t) |,
\] 
where $ B^0 $ is a Brownian bridge on $ [0,1] $. Again, $ \alpha_n $
can be replaced by our estimator if it is unknown, using a learning sample satisfying the assumptions of Theorem~\ref{ThLRVEst}.

\section*{Acknowledgements}

Parts of this paper have been written during a research stay of the first author as visiting professor at Universit\'e catholique de Louvain, Louvain-la-Neuve, a visit at the Department of Statistics of Seoul National University, Seoul, South Korea, and a sabbatical. He thanks Rainer von Sachs and Sangyeol Lee for their hospitality.
Rainer von Sachs gratefully acknowledges funding by contract “Projet d'Actions de Recherche Concert\' ees'' No. 12/17-045 of the ,,Communaut\'e fran\c caise de Belgique'' and by IAP research network Grant P7/06 of the Belgian government (Belgian Science Policy). He wants to thank the Institute of Stochastics of the RWTH Aachen (Prof. A. Steland and his group) for the warm hospitality during his sabbatical stay in autumn 2013, and part of this sabbatical has generously been funded by the Universit\'e catholique de Louvain.
We thank two anonymous referees for helpful comments which improved the presentation of the results.

\appendix

\section*{Proofs}


\begin{proof}[Proof of Lemma~\ref{TechnicalLemma}] 
The assertions follow from 
\ifthenelse{\boolean{extendedversion}}{(\ref{FProp1})--(\ref{FProp3}),}{}
\[
  \sup_{n \in \N} ( c_j^w )^2  = \sup_{n \in \N} \left( \sum_{\nu=1}^{d_n} w_\nu c_j^{(\nu)} \right)^2
    \le \sup_{n \in \N} \max_{1 \le j \le d_n} | c_{nj}^{\nu} |^2 \| \vecw_n \|_{\ell_1}^2
\]
and \cite[Remark 3.2]{Kouritzin1995}.
\end{proof}

\begin{proof}[Proof of Theorem~\ref{Th1}]
Notice that we have 
\begin{align*}
   D_{nk}( \vecv_n, \vecw_n )
& = \vecv_n' (\wh{\bfSigma}_{n,k} - \bfSigma_{n,k}) \vecw_n \\
& =
    \sum_{\nu, \mu =1}^{d_n} v_{\nu} w_{\mu} \sum_{i \le k} 
[ Y_i^{(\nu)} Y_i^{(\mu)} - E Y_i^{(\nu)} Y_i^{(\mu)} ] \\
  & =
  \sum_{i \le k} \left\{ 
\sum_{\nu, \mu=1}^{d_n} v_{\nu} w_{\mu} Y_i^{(\nu)} Y_i^{(\mu)} - 
\sum_{\nu, \mu=1}^{d_n} v_{\nu} w_{\mu} E Y_i^{(\nu)} Y_i^{(\mu)} 
\right\}
\end{align*}
leading to the representation
\begin{equation}
\label{ReprDnk}
  D_{nk}( \vecv_n, \vecw_n ) = \sum_{i\le k} [ Y_{ni}{(\vecv_n)}Y_{ni}{(\vecw_n)} - E Y_{ni}{(\vecv_n)} Y_{ni}{(\vecw_n)} ] 
\end{equation}
with linear processes
\begin{equation}
\label{DefLinProcInnerProd}
  Y_{ni}{(\vecv_n)} = \sum_{j=0}^\infty c_{nj}^{(v)} \epsilon_{i-j}, 
  \qquad
  Y_{ni}{(\vecw_n)} 
  = \sum_{j=0}^\infty c_{nj}^{(w)} \epsilon_{i-j},
\end{equation}
 w.r.t. $ \{ \epsilon_t \} $ given by the coefficients
\[
c_{nj}^{(v)} = \sum_{\nu=1}^{d_n} v_\nu c_{nj}^{(\nu)},
\qquad 
c_{nj}^{(w)} = \sum_{\nu=1}^{d_n} w_\nu c_{nj}^{(\nu)},
\]
for $ j \ge 0 $ and $ n \ge 1 $. We may now follow the method of proof of \cite{Kouritzin1995}, however, we have to take into
account that the above processes depend on $n$.

Let $ \calF_m = \sigma( \epsilon_i : i \le m ) $, $ m \ge 1 $. It is easy to check that, for any fixed $n \in \N $, the r.v.s.
\[
  M_m^{(n)}( \vecv_n, \vecw_n )
= \wt{f}_{0,0}^{(n)}( \vecv_n, \vecw_n ) \sum_{k=0}^m 
( \epsilon_k^2 - \sigma_k^2) + \sum_{k=0}^m \epsilon_k \sum_{l=1}^\infty \wt{f}_{l,0}^{(n)}( \vecv_n, \vecw_n ) \epsilon_{k-l},
\qquad m \ge 0,
\]
satisfy  $ E( M_m^{(n)}( \vecv_n, \vecw_n ) \, | \, \calF_{m-1} )
= M_{m-1}^{(n)}( \vecv_n, \vecw_n ) $, for $ m \ge 0 $, thus forming
a martingale array 
$
  \{ M_m^{(n)}( \vecv_n, \vecw_n ) : m \in \N, n \in \N \}
$
with associated martingale differences 
\begin{align*}
  &M_{n' + m'}^{(n)}( \vecv_n, \vecw_n ) - M_{m'}^{(n)}( \vecv_n, \vecw_n )  \\
 & \qquad =
\wt{f}_{0,0}^{(n)}( \vecv_n, \vecw_n ) \sum_{k=m'+1}^{n'+m'} 
( \epsilon_k^2 - \sigma_k^2) + \sum_{k=m'+1}^{n'+m'} \epsilon_k \sum_{l=1}^\infty \wt{f}_{l,0}^{(n)}( \vecv_n, \vecw_n ) \epsilon_{k-l},
\end{align*}
for $ n', m' \ge 0 $. Put
\begin{equation}
\label{DefDnm}
  D_{n',m'}^{(n)}( \vecv_n, \vecw_n ) = \sum_{k=m'+1}^{m'+n'} [Y_k{(\vecv_n)} Y_k{(\vecw_n)} - E Y_k{(\vecv_n)} Y_k{(\vecw_n)} ], \qquad m', n' \ge 0,
\end{equation}
and consider the decomposition
\[
  D_{n',m'}^{(n)}( \vecv_n, \vecw_n ) 
  = M_{n'+m'}^{(n)}( \vecv_n, \vecw_n ) - M_{m'}^{(n)}( \vecv_n, \vecw_n ) + R_{n',m'}^{(n)}( \vecv_n, \vecw_n ), \qquad m', n' \ge 0.
\]
In order to justify the approximation of  $ D_{n', m'}^{(n)}( \vecv_n, \vecw_n ) $ by the martingale differences defined above, it suffices to show that 
 $  \sup_n E [R_{n', m'}^{(n)}( \vecv_n, \vecw_n )]^2 $ tends to $0$ sufficiently fast, as $ n', m' \to \infty $. Using the representation 
\cite[(4.3)]{Kouritzin1995}, repeating the arguments in \cite{Kouritzin1995} leading to the bounds in (4.8), (4.9) and (4.10) therein and noting that those bounds are uniform in $ n \ge 1 $, we obtain
\[
  E( R_{n' m'}^{(n)}( \vecv_n, \vecw_n ) )^2 \stackrel{n',m'}{<<} (n')^{-1-\theta}.
\]
and, for each $ n \in \N $,
\[
  \| E[ (D_{m' n'}^{(n)}( \vecv_n, \vecw_n ))^2 \, | \, \calF_{m'} ] \|_1
<< (n')^{-1-\theta}.
\]
This approximation in $ L_2 $ with a rate allows for very general conditions for the validity of a strong approximation.
Again, we may follow the arguments given by \cite{Kouritzin1995}, by verifying the following sufficient conditions due to \cite{Philipp1986}. In terms of an array $ \xi_k^{(n)} $, $ k = 1, \dots, n $, of r.v.s., Philipp's result is as follows. Let $ \calG_m^{(n)} = \sigma( \xi_{ni} : i \le m ) $. If 
\begin{equation}
\label{DefS}
 S_{n',m'}^{(n)} = \sum_{k=m'+1}^{m'+n'} \xi_k^{(n)}, m', n' \ge 0,
\end{equation} 
satisfies
\begin{itemize}
  \item[(I)]   $ \| E( S_{n', m'}^{(n)} | \calG_{m'}^{(n)} ) \|_1 \stackrel{m',n'}{<<} (n')^{1/2-\varepsilon} $, a.s., for some $ \varepsilon > 0 $,
  \item[(II)]  there exists an $ \alpha_n^2 \ge 0 $ such that
$ \| E[ (S_{n',m'}^{(n)})^2 | \calG_{m'}^{(n)} ] - n' \alpha_n^2 \|_1 \stackrel{n',m'}{<<} (n')^{1-\varepsilon} $, a.s., for some $ \varepsilon > 0 $.
  \item[(III)] $ \sup_{k \ge 0} E | \xi_k^{(n)} |^{4+\delta} < \infty $ for some $ \delta > 0 $,
\end{itemize}
then there exists a process $ \{ \wt{S}_{n'}^{(n)}  : n' \ge 0 \} $ and a standard Brownian motion $ \{ \wt{B}_t^{(n)} : t \ge 0 \} $ on some probability space
$ ( \wt{\Omega}, \wt{\calF}, \wt{P} ) $, such that $ \{ \wt{S}_{n'}^{(n)}  : n' \ge 0 \} \stackrel{d}{=}  \{ S_{n',0}^{(n)}  : n' \ge 0 \} $ and for some $ \lambda > 0 $
\[
  | \wt{S}_{\lfloor t \rfloor}^{(n)} - \alpha_n \wt{B}_t^{(n)} | 
\stackrel{t}{<<} t^{1/2-\lambda},
\] 
for all $ t>0 $ $ \wt{P} $-a.s.  Putting, for fixed $ n \ge 1 $,  
\begin{equation}
\label{DefXi}
 \xi_k^{(n)} = \xi_k^{(n)}( \vecv_n, \vecw_n) = Y_k(\vecv_n) Y_k(\vecw_n) - E( Y_k(\vecv_n) Y_k(\vecw_n) ),
\end{equation}
$ \calG_m^{(n)} = \calF_m $ (since the $ \xi_k^{(n)} $ are $ \calF_k $--measurable) and repeating the arguments of \cite{Kouritzin1995},
we see that, by virtue of Assumption (A), (I)-(III) hold true, which establishes
the existence of a standard Brownian motion, $ B_n(t) $, $ t \in [0,\infty) $, such that for some constant $ C_n $ and some universal $ \lambda > 0 $
\[ 
  | D_{nt} - \alpha_n B_n(t) | \le C_n t^{1/2-\lambda},
\] 
for all $ t > 0 $, a.s. Denoting the standard Brownian motion on $ [0,1] $ associated to $ B_n $, 
$t \mapsto n^{-1/2} B_n( t n ) $, $ t \in [0,1] $, again by $B_n$, we obtain
\[
  \sup_{t \in [0,1]} | n^{-1/2} D_{n, \trunc{nt}} - \alpha_n B_n( \trunc{nt} / n ) | 
  \le C_n n^{-\lambda},
\]
which establishes (\ref{StrongApprox}) and (\ref{MainCLT}), provided $ C_n n^{-\lambda} = o(1) $.
It also follows that, for each fixed $n$, the conditional variance of $ M_{m'+n'}^{(n)} - M_{m'}^{(n)} $ satisfies
\begin{equation}
\label{CondVariance}
  \| E[ ( M_{m'+n'}^{(n)}( \vecv_n, \vecw_n ) - M_{m'}^{(n)}( \vecv_n, \vecw_n ) )^2 | \calF_{m'} ] - n' \alpha_n^2 \|_1 
\stackrel{n',m'}{<<} (n')^{1-\theta/2}
\end{equation}
and (cf. \cite[(4.22)]{Kouritzin1995})
\begin{equation}
\label{CondVar2}
  \| E[ ( D_{n',m'}^{(n)}(\vecv_n, \vecw_n ))^2 \, | \, \calF_{m'} ]
  - n' \alpha_n^2( \vecv_n, \vecw_n ) \|_1 
  \stackrel{n',m'}{<<} (n')^{1-\theta/2}
\end{equation}
as well as
\begin{equation}
\label{CondVar3}
  | E( D_{n',m'}^{(n)}( \vecv_n, \vecw_n  )^2 - n' \alpha_n^2( \vecv_n, \vecw_n ) |
  \stackrel{n',m'}{<<} (n')^{1-\theta/2}.
\end{equation}
The constants appearing in (\ref{CondVariance}), (\ref{CondVar2}) and (\ref{CondVar3}) depend on
the weighting vectors only through their $ \ell_1 $--norms, as a consequence of eqns (4.22)--(4.26) in \cite{Kouritzin1995}
and Lemma~\ref{TechnicalLemma}. 
\end{proof}

Recall that for a $K$--dimensional mean zero random vector $ \vecZ $ with 
variance-covariance matrix $ \matW $, say, the covariance operator
$ C(\vecu ) = E( \vecu'\vecZ \vecZ ) $, $ \vecu \in \R^K $, can be identified with the linear
mapping $ \R^K \mapsto L(\R^K; \R) $, $ \vecu \mapsto \vecu' \matW $, $ \vecu \in \R^K $, induced by the  variance-covariance matrix $ \matW $, where $ L(A,B) $ denotes the set of linear mappings $A \to B $.

\begin{proof}[Proof of Theorem~\ref{ThK}] 
Put
\[ 
  \vecD_{nk} = ( D_{nk}(j) )_{j=1}^K = ( \vecv_{nj}'( \wh{\bfSigma}_{n,k} - \bfSigma_{n,k} ) \vecw_{nj} )_{j=1}^K 
\] 
and notice that 
\[ \vecD_{nk} = \sum_{i\le k} \bfxi_{i}^{(n)}, \qquad \bfxi_i^{(n)} = ( Y_{ni}( \vecv_{nj} ) Y_{ni}( \vecw_{nj} ) - E  Y_{ni}( \vecv_{nj} ) Y_{ni}( \vecw_{nj} ) )_{j=1}^K.
\] 
Also put $ \vecS_{n',m'}^{(n)} = \sum_{k=m'+1}^{m'+n'} \bfxi_k^{(n)} $, $ m', n' \ge 0 $. For the Euclidean space $ \R^K $ equipped with the usual inner product and the induced vector $ \ell_2$--norm, the conditions (I) and (III) are easily checked. For instance, Jensen's inequality yields
\[
  E \| E( \vecS_{n',m'}^{(n)} | \calF_{m'} ) \|_{\ell_2} \le \sqrt{K} C ( n' )^{1-\theta} 
\]
for some constant $ C $ which does not depend on $ n', m', n $ and $ \{ ( \vecv_{nj}, \vecw_{nj} ) : j = 1, \dots, K \} $.
Introduce the conditional covariance operator 
\[ 
  C_{n',m'}^{(n)}( \vecu ) = E( \vecu ' \vecS_{n',m'}^{(n)}  \vecS_{n',m'}^{(n)} | \calF_{m'} ), \qquad \vecu \in \R^K.
\] 
and the covariance operator 
\[ 
  T^{(n)}( \vecu ) = E( \vecu' \vecB^{(n)} \vecB^{(n)} ), \qquad \vecu \in \R^K. 
\] 
Noting that $ T^{(n)}( \vecu ) = \sum_{j=1}^K u_j ( \Cov( B_{n1}, B_{nj} ), \dots, \Cov( B_{nK}, B_{nj} ) )' $, we have to check the remaining condition
\begin{itemize}
 \item[(II)]  $ E \| (n')^{-1} C_{n',m'}^{(n)} - C^{(n)} \| \stackrel{n',m'}{<<} (n')^{-\theta}$, for some covariance operator $C^{(n)} $.
\end{itemize}
Here the norm is the operator norm defined as $ \| L \| = \sup_{\vecu\in \R^K, \| \vecu \| = 1} | \vecu'L(\vecu) | $ for
a linear operator $ L : \R^K \to \R^K $. It follows that (II) holds true with $ C^{(n)} = T^{(n)} $, if for $ i = 1, \dots, K $
\[
  \| E[ (D_{n',m'}^{(n)}(i))^2 | \calF_{m'} ] - n' \alpha_n(\vecv_{ni}, \vecw_{ni} ) \|_{L_1} \stackrel{m'}{<<}
  (n')^{1-\theta/2}
\]
and for $ i, j = 1, \dots, K $ with $ i \not= j $
\begin{equation}
\label{L1ApproxBetas}
  \| E[ D_{n',m'}^{(n)}(i) D_{n',m'}^{(n)}(j) | \calF_{m'} ] - n'  \beta_n( \vecv_{ni}, \vecw_{ni}, \vecv_{nj}, \vecw_{nj} ) \|_{L_1} 
\stackrel{‚m'}{<<}
  (n')^{1-\theta/2}.
\end{equation}
The last fact follows by a  lengthy but straightforward calculation using (\ref{Ass5}).
\end{proof}

\begin{proof}[Proof of Theorem~\ref{Th2}]
The proof is similar to the proof of Theorem~\ref{Th1} and \cite{Kouritzin1995} having observed the following crucial facts. We have
\[
  D_k(  \{ \vecv_n, \vecw_n \} ) =
  \sum_{i \le k} [Y_i( \{ \vecv_n \} ) Y_i( \{ \vecw_n \} ) - E( Y_i( \{ \vecv_n \} ) Y_i( \{ \vecw_n \} ) ) ],
\]
for the linear processes $ Y_i( \{ \vecv_n \} ) = \sum_{j=0}^\infty c_j^{(v)} \epsilon_{i-j} $, $ i \ge 1 $, and
$  Y_i( \{ \vecv_n \} ) = \sum_{j=0}^\infty c_j^{(w)} \epsilon_{i-j} $ with coefficients
$ c_j^{(v)} = \sum_{n=1}^\infty \lambda_n \sum_{\nu=1}^{d_n} v_{\nu} c_{nj}^{(\nu)} $ 
and
$ c_j^{(w)} = \sum_{n=1}^\infty \lambda_n \sum_{\nu=1}^{d_n} w_{\nu} c_{nj}^{(\nu)} $ for $ j \ge 0 $,
which do not depend on the sample size. Since $ \sum_n | \lambda_n | < \infty $,
\[
  ( c_j^{(v)} )^2 << \left( \sum_{n=1}^\infty |\lambda_n| \right)^2 \sup_{n \ge 1} 
  \left( \sum_{\nu=1}^{d_n} v_\nu c_{nj}^{(\nu)} \right)^2 <<(j \vee 1)^{-3/2-\theta/2},
\]
i.e. Assumption (A) is satisfied. Hence the result 
\end{proof}

\ifthenelse{\boolean{extendedversion}}{
\begin{proof}[Proof of Theorem~\ref{ThLRVEst}] 
First observe that $ Y_k^{(v)} Y_k^{(w)} $, $ k \ge 1 $, as well as $ Y_k^{(v_r)} Y_k^{(w_r)} Y_{k+h}^{(v_s)} Y_{k+h}^{(w_s)} $, $ k \ge 1 $, are strictly stationary for any fixed $ r, s $ and $ h $. Their dependence on $d$ will be suppressed in notation. It follows from the proof of Theorem~\ref{Th1} and \cite[p.~351]{Kouritzin1995} that $ \alpha^2 $ can be represented as
\[
  \alpha^2 = \lim_{N \to \infty} \Var \left( \frac{1}{\sqrt{N}}
  \sum_{k=0}^N [Y_k^{(v)} Y_k^{(w)} - E( Y_1^{(v)} Y_1^{(w)} ) ] \right), 
\] 
and therefore $ \alpha^2 $ is the long-run variance parameter associated to the time series 
\[
  \xi_k =  Y_k^{(v)} Y_k^{(w)} - E( Y_1^{(v)} Y_1^{(w)} ),  \qquad k \ge 1,
\] 
and $ \wh{\alpha}^2_n $ is the Bartlett type estimator calculated from the first $ n $ observations. Analogously, by virtue of (\ref{L1ApproxBetas}), 
\begin{align*}
  \beta^2(r,s) &= \lim_{N \to \infty} 
  E 
    \left(
      \frac{1}{\sqrt{N}} \sum_{k=1}^N 
      \xi_k( r ) \right)
    \left(
      \frac{1}{\sqrt{N}} \sum_{k=1}^N 
      \xi_k(s) 
    \right) \\
    &= E[ \xi_1(r) \xi_1(s) ] + 2 \lim_{N \to \infty} 
    \sum_{h=1}^N \frac{N-h}{N} E[ \xi_1(r) \xi_{1+h}(s) ],
\end{align*}
where
\[
  \xi_k(\ell) = Y_k(\vecv_\ell) Y_k(\vecw_\ell) - E[Y_1(\vecv_\ell) Y_1(\vecw_\ell)  ], \qquad k \ge 1,
\]
for $ \ell = 1, 2, \dots $. 
Again, we omit the dependence on $d$, but will indicate it for relevant derived quantities when appropriate. 
Put
\begin{align*}
  \Gamma_h^{(2)}(r,s) &= \Gamma_h^{(2)}(r,s;d) =E( Y_1^{(v_r)} Y_1^{(w_r)} Y_{1+|h|}^{(v_s)} Y_{1+|h|}^{(w_s)} ), \\
  \Gamma_h &= \Gamma_h(r,s) = \Gamma_h(r,s;d) = E( \xi_1(r) \xi_{1+h}(s) ),
\end{align*}
for $ h \in \Z $. Observe that 
$ E ( Y_1^{(v_r)} Y_1^{(w_r)} ) E( Y_{1+h}^{(v_s)} Y_{1+h}^{(w_s)} )
= \sigma^4 \prod_{z \in \{r,s\}} \left( \sum_{j=0}^\infty c_j^{(v_z)} c_j^{(w_z)} \right) $.
Further, $ \Gamma_h^{(2)}(r,s) $ is the sum over all terms
$ c_j^{(v_r)} c_k^{(w_r)} c_l^{(v_s)} c_m^{(w_s)} 
  E( \epsilon_{i-j} \epsilon_{i-k} \epsilon_{i+h-l} \epsilon_{i+h-m} ) $
  such that either all four indices $ a=i-j, b=i-k, c= i+h-l, d=i+h-m $ are equal,
  denoted by $ \Gamma_{h,1}^{(2)}(r,s) $, plus the sum of all terms
  such that $ \{ a, b, c, d \} = \{ A, B \} $ with $A \not= B$, denoted
  by $ \Gamma_{h,2}^{(2)}(r,s) $.
  Using
  \[
  | c_j^{(v_r)} | \le \| \vecv_r \|_{\ell_1} \sup_{1 \le \nu} | c_j^{(\nu)} |
  = O( \| \vecv_r \|_{\ell_1} (j \vee 1)^{-(1+\delta)} )
  \]
  and
  \[
  | c_{j+h}^{(v_r)} | \le \| \vecv \|_{\ell_1} \sup_{1 \le \nu} \sup_{0 \ge j} | c_{j+h}^{(\nu)} | = O( \| \vecv_r \|_{\ell_1}  h^{-(1+\delta)} ),
  \]
  for $ j \ge 0 $ and $ h \ge 1$,  the former case leads to 
  $ 
    \Gamma_{h,1}^{(2)}(r,s) 
    =  E( \epsilon_1^4 ) \sum_{j=0}^\infty c_j^{(v_r)} c_j^{(w_r)} c_{j+h}^{(v_s)} c_{j+h}^{(w_s)} 
  =    
  O( h^{-2(1+\delta)} ) 
  $, uniformly in $d \in \N $ and $ 1 \le r,s$. To discuss $ \Gamma_{h,2}^{(2)}(r,s;d) $ consider the subcase
  $ j = k, l = m, l \not= h+j $ corresponding to the sum
  \[
    \sigma^4 \sum_{j=0}^\infty c_j^{(v_r)} c_j^{(w_r)}
    \sum_{l=0,l\not= j+h} c_l^{(v_s)} c_l^{(w_s)}
    = E( Y_1^{(v_r)} Y_1^{(w_r)} ) E( Y_1^{(v_s)} Y_1^{(w_s)} )
    - \overline{\Gamma}_h^{(2)}(r,s;d),
  \]
  where 
  $ \sup_{1 \le r,s} \sup_{d \in \N} |\overline{\Gamma}_h^{(2)}(r,s;d)| << \sum_{j=0}^\infty (j\vee 1)^{-2(1+\delta)} h^{-2(1+\delta)} $, leading to \[ 
   \sup_{1 \le r,s} \sup_{d \in \N} \sum_{h=1}^\infty | \overline{\Gamma}_h^{(2)}(r,s;d) | < \infty. \] 
   The other subcases are straightforward, such that we arrive at
\begin{align*}
  \sup_{1 \le r,s} \sup_{d \in \N}
    \sum_{h=1}^\infty |\Gamma_h(r,s;d)|
    & =  \sup_{1 \le r,s} \sup_{d \in \N} \sum_{h=1}^\infty
    | \Gamma_h^{(2)}(r,s;d) - E( Y_1^{(v_r)}  E( Y_1^{(w_r)} ) E(Y_{1+h}^{(v_s)} Y_{1+h}^{(w_s)} ) | < \infty.
\end{align*}
Hence
\begin{equation}
\label{UniformSeriesRepr}
  \sup_{1 \le r, s } \sup_{d \in \N} |\beta^2(d;r,s)| \le \sup_{1 \le r, s } \sup_{d \in \N} \sum_{h \in \Z} | \Gamma_h(d;r,s) | < \infty.
\end{equation} 
Next introduce the coupling dependence measure
\[
  \delta_p(\{ Z_i : i \in \N_0 \}, n) = \| Z_n- Z_n' \|_{L_p} 
\]
$ p \ge 1 $, for a time series $ Z_n = Z(\epsilon_n, \epsilon_{n-1}, \dots ) $,
where $ Z_n' = Z( \epsilon_n, \dots, \epsilon_1, \epsilon_0', \epsilon_{-1}, \dots ) $ with $ \epsilon_0 \stackrel{d}{=} \epsilon_0' $ such that $ \epsilon_0' $ is independent from $ \{ \epsilon_k \} $. Since $ Y_i^{(v)} $ is a causal linear process with coefficients
$ c_j^{(v)} = \sum_{\nu=1}^d v_\nu c_j^{(\nu)} $, we have
\[
  \| Y_k^{(v)} \|_{L_8} \le 
    \| \epsilon_0 \|_{L_8} \| \vecv \|_{\ell_1} \sum_{j=0}^\infty \sup_{1 \le \nu} | c_j^{(\nu)} | \le c 
\]
and by (\ref{DecayCoeff})
\[
  \delta_8( \{ Y_i^{(v)} \}, k ) = \| Y_k - Y_k' \|_{L_8}
  \le E|\epsilon_1|^8 \| \vecv \|_{\ell_1} \sup_{1 \le \nu} | c_k^{(\nu)} |
  = O( (k \vee 1)^{-(1+\delta)} )
\]
such that $ \sum_{k=0}^\infty \delta_8( \{ Y_i^{(v)} \}, k ) \le C < \infty $ for constants $ c, C < \infty $ not depending on $d$ and uniformly
over $ \| \vecv \|_{\ell_1} \le C_{v,w} $. Further, (using
$ [ E( X^4 Y^4 ) ]^{1/4} \le [ ( E X^8 )^{1/2} ( E Y^8 )^{1/2} ]^{1/4} = [ E X^8 ]^{1/8} [ E Y^8 ]^{1/8} $),
\begin{align*}
  \delta_4( \{ \xi_i \}, k ) 
  & = \| Y_k^{(v)} Y_k^{(w)} - (Y_k^{(v)} Y_k^{(w)})' \|_{L_4} \\
  & \le \| Y_k^{(v)} \|_{L_8} \| Y_{k}^{(w)} - Y_{k}^{(w)}{}' \|_{L_8} 
    + \| Y_{k}^{(w)} \|_{L_8} \| Y_k^{(v)} - Y_k^{(v)}{}' \|_{L_8} \\
  & = O( \delta_8( \{ Y_i^{(v)} \}, k ) + \delta_8( \{ Y_i^{(w)} \}, k ) )
\end{align*}
leading to $ \sum_{k=0}^\infty \delta_4(\{ \xi_i \}, k ) \le C_1 < \infty $ for some constant $C_1$, uniformly over $ d \in \N $
and $ \| \vecv \|_{\ell_1}, \| \vecw \|_{\ell_1} \le C_{v,w} $. Lastly,
\begin{align*}
  \delta_2( \{ \xi_i(r) \xi_{i+h}(s) \}, k ) 
  & = \| Y_k^{(v_r)} Y_k^{(w_r)} Y_{k+h}^{(v_s)} Y_{k+h}^{(w_s)} - (Y_k^{(v_r)} Y_k^{(w_r)} Y_{k+h}^{(v_s)} Y_{k+h}^{(w_s)})' \|_{L_2} \\
  & \le 
  \|  Y_k^{(v_r)} Y_k^{(w_r)} \|_{L_4} 
  \| Y_{k+h}^{(v_s)} Y_{k+h}^{(w_s)}  - ( Y_{k+h}^{(v_s)} Y_{k+h}^{(w_s)} )' \|_{L_4} \\
  & \qquad
  + \| Y_{k+h}^{(v_s)} Y_{k+h}^{(w_s)} \|_{L_4}
     \| Y_k^{(v_r)} Y_k^{(w_r)}  - (Y_k^{(v_r)} Y_k^{(w_r)} )' \|_{L_4} \\
  &= O( \delta_4( \{ \xi_i \}, k) )  \\
  & = O( \delta_8( \{ Y_i^{(v)} \}; k ) + \delta_8( \{ Y_i^{(w)} \}; k ) ),
\end{align*}
leading to $ \sum_{k=0}^\infty \delta_2( \{ \xi_i(r) \xi_{i+h}(s) \}, k ) < C_2 $ for a constant $C_2$, uniformly over $d \in \N$ and $ 1 \le r, s $. Define 
\[
  \wt{\Gamma}_h(r,s) = \frac{1}{n} \sum_{i=1}^{n-h} \xi_i(r) \xi_{i+h}(s).
\] 
By virtue of \cite[Th.~1]{Wu2007}, we obtain
\begin{equation}
\label{L2ConvWu}
  \sup_{d \in \N} E (n [ \wt{\Gamma}_h(r,s;d) - E( \wt{\Gamma}_h(r,s;d) ) ] )^2  \le C_3 (n-h)
\end{equation}
for some constant $C_3< \infty $ not depending on $ h $ or $ m $,
uniformly over $ \| \vecv \|_{\ell_1} \le C_{v,w} $, such that
\[
  \sup_{1 \le r, s} \sup_{d \in \N} \max_{|h| \le m_n} \| \wt{\Gamma}_h(r,s; d) - E( \wt{\Gamma}_h(r,s; d) ) \|_{L_2}
  \le C_4 n^{-1/2},
\]
for some constant $ C_4 < \infty $. Observe that
$ \wh{\Gamma}_h(r,s) = \frac{1}{n} \sum_{i=1}^{n-h} ( \xi_i(r) - \overline{\xi}_n(r) )( \xi_{i+h}(s) - \overline{\xi}_n(s) ) $, where
$ \overline{\xi}_n(\ell) = n^{-1} \sum_{j=1}^n \xi_j(\ell) $, $ \ell = 1, 2, \dots $, and
\[
  n[ \wh{\Gamma}_h(r,s) - \wt{\Gamma}_h(r,s) ] =
  - \overline{\xi}_n(r) \sum_{j=1}^{n-h} \xi_{j+h}(s) 
  - \overline{\xi}_n(s) \sum_{j=1}^{n-h} \xi_j(r) 
  + \overline{\xi}_n(r) \sum_{j=1}^n \xi_j(s).
\]
Using the Cauchy-Schwarz inequality, we thus obtain
\begin{align*}
  n E |\wh{\Gamma}_h(r,s) - \wt{\Gamma}_h(r,s) |
  &\le 
  \| \overline{\xi}_n(r) \|_{L_2} \left\| \sum_{j=1}^{n-h} \xi_{j+h}(s) \right\|_{L_2}
  + \| \overline{\xi}_n(s) \|_{L_2} \left\| \sum_{j=1}^{n-h} \xi_{j}(r) \right\|_{L_2}  \\
  & \qquad
  + \| \overline{\xi}_n(r) \|_{L_2} \left\|  \sum_{i=1}^n \xi_i(s) \right\|_{L_2}
= O(1)
\end{align*}
uniformly over $ h \le m $, $ d \in \N $ and $ 1 \le r, s $. It follows that
\[
  \sup_{1 \le r, s} \sup_{d \in \N} m \max_{|h| \le m}  E | \wh{\Gamma}_h(r,s;d) - \wt{\Gamma}_h(r,s;d) | =  O( m/n ) = o(1),
\]
which implies
\begin{align*}
  & \sup_{1 \le r, s}  \sup_{d \in \N} E \left| 
    \sum_{|h| \le m} w_{mh} \wh{\Gamma}_h(r,s;d) - 
    \sum_{|h| \le m} w_{mh} \wt{\Gamma}_h(r,s;d) 
    \right| \\
    & \qquad \le \sup_{1 \le r, s} \sup_{d \in \N} \max_{|h| \le m} E| \wh{\Gamma}_h(r,s;d) - \wt{\Gamma}_h(r,s;d) |
      \sum_{|h| \le m} w_{mh} \\
    & \qquad \le 2 W \sup_{1 \le r, s}  \sup_{d \in \N} m \max_{|h| \le m} E| \wh{\Gamma}_h(r,s;d) - \wt{\Gamma}_h(r,s;d) | = o(1),
\end{align*}
as $ n \to \infty $. Hence it suffices to show the result for
\[
  \wt{\beta}^2_n(r,s;d) = \sum_{|h| \le m} w_{mh} \wt{\Gamma}_h(r,s).
\]
Using the representation $ \beta^2(r,s;d) = \sum_{h \in \Z} \Gamma_h(r,s;d) $, we obtain the decomposition
\[
  \wt{\beta}_n^2(r,s;d) - \beta^2(r,s;d)
  = A_n(r,s;d) + B_n(r,s;d) + C_n(r,s;d) + D_n(r,s;d),
\]
where
\begin{align*}
  A_n(r,s;d) & = \sum_{|h| \le m} w_{mh} [\wt{\Gamma}_h(r,s;d) - E( \wt{\Gamma}_h(r,s;d) ) ], \\
  B_n(r,s;d) &= \sum_{|h| \le m} w_{mh} [ E( \wt{\Gamma}_h(r,s;d) ) - \Gamma_h(r,s;d) ], \\
  C_n(r,s;d) & = \sum_{|h| \le m} [ w_{mh} - 1] \Gamma_h(r,s;d),  \qquad
  D_n(r,s;d) = - \sum_{|h| > m} \Gamma_h(r,s;d).
\end{align*}
First observe that $ \sup_{1 /le r,s} \sup_{d \in \N} | D_n(r,s;d) |
= o(1) $ by (\ref{UniformSeriesRepr}), and, by Fubini and (\ref{L2ConvWu}), 
\begin{align*}
  \sup_{1 \le r, s} \sup_{d \in \N} E | A_n(r,s;d) | 
  &\le 
  \sup_{1 \le r, s} \sup_{d \in \N}  \int_{\Z} w_{mh} E | \wt{\Gamma}_h(r,s;d) - E \wt{\Gamma}_h(r,s;d) | \eins( |h| \le m) \, d \pi(h)  \\
  & \le 
  2 W m \sup_{1 \le r, s} \sup_{d \in \N}  \max_{|h| \le m} E | \wt{\Gamma}_h(r,s;d) - E \wt{\Gamma}_h(r,s;d) | \\
  & = O( m/n^{1/2} ) = o(1),
\end{align*}
as $ n \to \infty $, since $ m^2/n = o(1) $ by assumption. Here $ d \pi $ denotes the counting measure on $ \Z $. Further, 
\[
  E |B_n(r,s;d)| \le \int w_{mh} \left| \frac{n-h}{n} - 1 \right| | \Gamma_h(r,s;d) | \eins( |h| \le m ) \, d \pi(h),
\]
where the integrand is $ o(1) $ point-wise in $h$ and bounded by the (uniformly over  $ d \in \N $ and $ 1 \le r, s $) $ \pi $--integrable function $ 2 W | \Gamma_h(r,s;d) | $, such that $ \sup_{1 \le r, s} \sup_{d \in \N} E|B_n(r,s;d)| = o(1) $, as $ n \to \infty $, follows. Similiarly, $ \sup_{1 \le r, s} E |C_n(r,s;d)| = o(1) $, as $ n \to \infty $, uniformly in $d \in \N $, by (\ref{UniformSeriesRepr}) and (W1). Hence the assertion follows.
\end{proof}
}
{
\begin{proof}[Proof of Theorem~\ref{ThLRVEst}] 
First observe that $ Y_k^{(v)} Y_k^{(w)} $, $ k \ge 1 $, as well as $ Y_k^{(v_r)} Y_k^{(w_r)} Y_{k+h}^{(v_s)} Y_{k+h}^{(w_s)} $, $ k \ge 1 $, are strictly stationary for any fixed $ r, s $ and $ h $. Their dependence on $d$ will suppressed in notation. It follows from the proof of Theorem~\ref{Th1} and \cite[p.~351]{Kouritzin1995} that $ \alpha^2 $ can be represented as
\[
  \alpha^2 = \lim_{N \to \infty} \Var \left( \frac{1}{\sqrt{N}}
  \sum_{k=0}^N [Y_k^{(v)} Y_k^{(w)} - E( Y_1^{(v)} Y_1^{(w)} ) ] \right), 
\] 
and therefore $ \alpha^2 $ is the long-run variance parameter associated to the time series 
$
  \xi_k =  Y_k^{(v)} Y_k^{(w)} - E( Y_1^{(v)} Y_1^{(w)} ),  k \ge 1,
$ 
and $ \wh{\alpha}^2_n $ is the Bartlett type estimator calculated from the first $ n $ observations. Analogously, by virtue of (\ref{L1ApproxBetas}), 
\begin{align*}
  \beta^2(r,s) 
    &= E[ \xi_1(r) \xi_1(s) ] + 2 \lim_{N \to \infty} 
    \sum_{h=1}^N \frac{N-h}{N} E[ \xi_1(r) \xi_{1+h}(s) ],
\end{align*}
where
$
  \xi_k(\ell) = Y_k(\vecv_\ell) Y_k(\vecw_\ell) - E[Y_1(\vecv_\ell) Y_1(\vecw_\ell)  ] 
$, $ k \ge 1 $, for $ \ell = 1, 2, \dots $. 
  Put $ \Gamma_h = \Gamma_h(r,s) = \Gamma_h(r,s;d) = E( \xi_1(r) \xi_{1+|h|}(s) ) $, $ h \in \Z $.
  Using
  $ 
  | c_j^{(v_r)} | \le \| \vecv_r \|_{\ell_1} \sup_{1 \le \nu} | c_j^{(\nu)} |
  = O( \| \vecv_r \|_{\ell_1} (j \vee 1)^{-(1+\delta)} )
  $
  and
  $
  | c_{j+h}^{(v_r)} | \le \| \vecv \|_{\ell_1} \sup_{1 \le \nu}  \sup_{0 \le h}| c_{j+h}^{(\nu)} | = O( \| \vecv_r \|_{\ell_1}  h^{-(1+\delta)} ),
  $, for $ h \ge 1 $, a lengthy but straightforward calculation shows that 
\begin{equation}
\label{UniformSeriesRepr}
  \sup_{1 \le r, s } \sup_{d \in \N} |\beta^2(r,s;d)| \le \sup_{1 \le r, s } \sup_{d \in \N} \sum_{h \in \Z} | \Gamma_h(r,s;d) |  < \infty.
\end{equation} 
Next introduce the coupling dependence measure
\[
  \delta_p(\{ Z_i : i \in \N_0 \}, n) = \| Z_n- Z_n' \|_{L_p} 
\]
$ p \ge 1 $, for a time series $ Z_n = Z(\epsilon_n, \epsilon_{n-1}, \dots ) $,
where $ Z_n' = Z( \epsilon_n, \dots, \epsilon_1, \epsilon_0', \epsilon_{-1}, \dots ) $ with $ \epsilon_0 \stackrel{d}{=} \epsilon_0' $ such that $ \epsilon_0' $ is independent from $ \{ \epsilon_k \} $. Since $ Y_i^{(v)} $ is a causal linear process with coefficients
$ c_j^{(v)} = \sum_{\nu=1}^d v_\nu c_j^{(\nu)} $, we have
$
  \| Y_k^{(v)} \|_{L_8} \le 
    \| \epsilon_0 \|_{L_8} \| \vecv \|_{\ell_1} \sum_{j=0}^\infty \sup_{1 \le \nu} | c_j^{(\nu)} | \le c 
$
and by (\ref{DecayCoeff})
\[
  \delta_8( \{ Y_i^{(v)} \}, k ) = \| Y_k - Y_k' \|_{L_8}
  \le E|\epsilon_1|^8 \| \vecv \|_{\ell_1} \sup_{1 \le \nu} | c_k^{(\nu)} |
  = O( (k\vee 1)^{-(1+\delta)} )
\]
such that $ \sum_{k=0}^\infty \delta_8( \{ Y_i^{(v)} \}, k ) \le C < \infty $, for constants $ c, C < \infty $ not depending on $d$ and uniformly
over $ \| \vecv \|_{\ell_1} \le C_{v,w} $. Further
\begin{align*}
  \delta_4( \{ \xi_i \}, k ) 
  & = \| Y_k^{(v)} Y_k^{(w)} - (Y_k^{(v)} Y_k^{(w)})' \|_{L_4} \\
  & \le \| Y_k^{(v)} \|_{L_8} \| Y_{k}^{(w)} - Y_{k}^{(w)}{}' \|_{L_8} 
    + \| Y_{k}^{(w)} \|_{L_8} \| Y_k^{(v)} - Y_k^{(v)}{}' \|_{L_8} \\
  & = O( \delta_8( \{ Y_i^{(v)} \}, k ) + \delta_8( \{ Y_i^{(w)} \}, k ) )
\end{align*}
leading to $ \sum_{k=0}^\infty \delta_4(\{ \xi_i \}, k ) \le C_1 < \infty $ for some constant $C_1$, uniformly over $ d \in \N $
and $ \| \vecv \|_{\ell_1}, \| \vecw \|_{\ell_1} \le C_{v,w} $. Lastly,
analogously we obtain
\begin{align*}
  \delta_2( \{ \xi_i(r) \xi_{i+h}(s) \}, k ) 
  & = O( \delta_8( \{ Y_i^{(v)} \}; k ) + \delta_8( \{ Y_i^{(w)} \}; k ) ) 
\end{align*}
leading to $ \sum_{k=0}^\infty \delta_2( \{ \xi_i(r) \xi_{i+h}(s) \}, k ) < C_2 $ for a constant $C_2$, uniformly over $d \in \N$ and $ 1 \le r, s $. Define 
\[
  \wt{\Gamma}_h(r,s) = \frac{1}{n} \sum_{i=1}^{n-h} \xi_i(r) \xi_{i+h}(s).
\] 
Due to the above estimates of the dependence coupling measures
we can apply \cite[Th.~1]{Wu2007} and obtain
\begin{equation}
\label{L2ConvWu}
  \sup_{d \in \N} E (n [ \wt{\Gamma}_h(r,s;d) - E( \wt{\Gamma}_h(r,s;d) )] )^2  \le C_3 (n-h)
\end{equation}
for some constant $C_3< \infty $ not depending on $ h $ or $ m $,
uniformly over $ \| \vecv \|_{\ell_1} \le C_{v,w} $, such that
\[
  \sup_{1 \le r, s} \sup_{d \in \N} \max_{|h| \le m_n} \| \wt{\Gamma}_h(r,s; d) - E( \wt{\Gamma}_h(r,s; d) ) \|_{L_2}
  \le C_4 n^{-1/2},
\]
for some constant $ C_4 < \infty $. Observe that
$ \wh{\Gamma}_h(r,s) = \frac{1}{n} \sum_{i=1}^{n-h} ( \xi_i(r) - \overline{\xi}_n(r) )( \xi_{i+h}(s) - \overline{\xi}_n(s) ) $, where
$ \overline{\xi}_n(\ell) = n^{-1} \sum_{j=1}^n \xi_j(\ell) $, $ \ell = 1, 2, \dots $, and
\[
  n[ \wh{\Gamma}_h(r,s) - \wt{\Gamma}_h(r,s) ] =
  - \overline{\xi}_n(r) \sum_{j=1}^{n-h} \xi_{j+h}(s) 
  - \overline{\xi}_n(s) \sum_{j=1}^{n-h} \xi_j(r) 
  + \overline{\xi}_n(r) \sum_{j=1}^n \xi_j(s).
\]
Using the Cauchy-Schwarz inequality, we thus obtain
$
  n E |\wh{\Gamma}_h(r,s) - \wt{\Gamma}_h(r,s) |
= O(1)
$
uniformly over $ h \le m $, $ d \in \N $ and $ 1 \le r, s $. It follows that
\[
  \sup_{1 \le r, s} \sup_{d \in \N} m \max_{|h| \le m}  E | \wh{\Gamma}_h(r,s;d) - \wt{\Gamma}_h(r,s;d) | =  O( m/n ) = o(1),
\]
which implies, by boundedness of the weights,
\begin{align*}
  & \sup_{1 \le r, s}  \sup_{d \in \N} E \left| 
    \sum_{|h| \le m} w_{mh} \wh{\Gamma}_h(r,s;d) - 
    \sum_{|h| \le m} w_{mh} \wt{\Gamma}_h(r,s;d) 
    \right| = o(1),
\end{align*}
as $ n \to \infty $. Hence it suffices to show the result for
$
  \wt{\beta}^2_n(r,s;d) = \sum_{|h| \le m} w_{mh} \wt{\Gamma}_h(r,s).
$
Using the representation $ \beta^2(r,s;d) = \sum_{h \in \Z} \Gamma_h(r,s;d) $, we obtain the decomposition
\[
  \wt{\beta}_n^2(r,s;d) - \beta^2(r,s;d)
  = A_n(r,s;d) + B_n(r,s;d) + C_n(r,s;d) + D_n(r,s;d),
\]
where
\begin{align*}
  A_n(r,s;d) & = \sum_{|h| \le m} w_{mh} [\wt{\Gamma}_h(r,s;d) - E( \wt{\Gamma}_h(r,s;d) ) ], \\
  B_n(r,s;d) &= \sum_{|h| \le m} w_{mh} [ E( \wt{\Gamma}_h(r,s;d) ) - \Gamma_h(r,s;d) ], \\
  C_n(r,s;d) & = \sum_{|h| \le m} [ w_{mh} - 1] \Gamma_h(r,s;d),  \qquad
  D_n(r,s;d) = - \sum_{|h| > m} \Gamma_h(r,s;d).
\end{align*}
First observe that $ \sup_{1 \le r,s} \sup_{d \in \N} |D_n(r,s;d)| = o(1) $  by (\ref{UniformSeriesRepr}). Fubini and (\ref{L2ConvWu}) and yield
\begin{align*}
  \sup_{1 \le r, s} \sup_{d \in \N} E | A_n(r,s;d) | 
  &\le 
  \sup_{1 \le r, s} \sup_{d \in \N}  \int_{\Z} w_{mh} E | \wt{\Gamma}_h(r,s;d) - E \wt{\Gamma}_h(r,s;d) | \eins( |h| \le m) \, d \pi(h)  \\
  & \le 
  2 W m \sup_{1 \le r, s} \sup_{d \in \N}  \max_{|h| \le m} E | \wt{\Gamma}_h(r,s;d) - E \wt{\Gamma}_h(r,s;d) | = o(1) 
\end{align*}
as $ n \to \infty $, since $ m^2/n = o(1) $ by assumption. Here $ d \pi $ denotes the counting measure on $ \Z $. Further, 
\[
  E |B_n(r,s;d)| \le \int w_{mh} | (n-h)/n - 1 | | \Gamma_h(r,s;d) | \eins( |h| \le m ) \, d \pi(h),
\]
where the integrand is $ o(1) $ point-wise in $h$ and bounded by the (uniformly over  $ d \in \N $ and $ 1 \le r, s $) $ \pi $--integrable function $ 2 W | \Gamma_h(r,s;d) | $, such that $ \sup_{1 \le r, s} \sup_{d \in \N} E|B_n(r,s;d)| = o(1) $, as $ n \to \infty $, follows. Similiarly, $ \sup_{1 \le r, s} E |C_n(r,s;d)| = o(1) $, as $ n \to \infty $, uniformly in $d \in \N $, by (\ref{UniformSeriesRepr}) and (W1). Hence the assertion follows.
\end{proof}
}



\begin{thebibliography}{}

\bibitem[\protect\citeauthoryear{Allison, Cui, Page, and Sabripour}{Allison
  et~al.}{2006}]{AllisonEtAl2006}
Allison, D., X.~Cui, G.~Page, and M.~Sabripour (2006).
\newblock Microarray data analysis: {F}rom disarray to consolidation and
  consensus.
\newblock {\em Nature Reviews Genetics\/}~{\em 7}, 55--65.

\bibitem[\protect\citeauthoryear{Andrews}{Andrews}{1991}]{Andrews1991}
Andrews, D. W.~K. (1991).
\newblock Heteroskedasticity and autocorrelation consistent covariance matrix
  estimation.
\newblock {\em Econometrica\/}~{\em 59\/}(3), 817--858.

\bibitem[\protect\citeauthoryear{Aue, H{\"o}rmann, Horv{\'a}th, and
  Reimherr}{Aue et~al.}{2009}]{AueHoermannHorvath2009}
Aue, A., S.~H{\"o}rmann, L.~Horv{\'a}th, and M.~Reimherr (2009).
\newblock Break detection in the covariance structure of multivariate time
  series models.
\newblock {\em Ann. Statist.\/}~{\em 37\/}(6B), 4046--4087.

\bibitem[\protect\citeauthoryear{Aue and Horv{\'a}th}{Aue and
  Horv{\'a}th}{2013}]{HorvathAue2013}
Aue, A. and L.~Horv{\'a}th (2013).
\newblock Structural breaks in time series.
\newblock {\em J. Time Series Anal.\/}~{\em 34\/}(1), 1--16.

\bibitem[\protect\citeauthoryear{Bickel and Levina}{Bickel and
  Levina}{2008}]{BickelLevina2008}
Bickel, P.~J. and E.~Levina (2008).
\newblock Covariance regularization by thresholding.
\newblock {\em Ann. Statist.\/}~{\em 36\/}(6), 2577--2604.

\bibitem[\protect\citeauthoryear{Biedermann, Nagel, Munk, Holzmann, and
  Steland}{Biedermann et~al.}{2006}]{BiedermannEtAl2006}
Biedermann, S., E.~Nagel, A.~Munk, H.~Holzmann, and A.~Steland (2006).
\newblock Tests in a case-control design including relatives.
\newblock {\em Scand. J. Statist.\/}~{\em 33\/}(4), 621--635.

\bibitem[\protect\citeauthoryear{B\"ohm and von Sachs}{B\"ohm and von
  Sachs}{2008}]{BoehmSachs2008}
B\"ohm, H. and R.~von Sachs (2008).
\newblock Structural shrinkage of nonparametric spectral estimators for
  multivariate time series.
\newblock {\em Electronic Journal of Statistics\/}~{\em 2}, 696--721.

\bibitem[\protect\citeauthoryear{B{\"o}hm and von Sachs}{B{\"o}hm and von
  Sachs}{2009}]{BoehmSachs2009}
B{\"o}hm, H. and R.~von Sachs (2009).
\newblock Shrinkage estimation in the frequency domain of multivariate time
  series.
\newblock {\em J. Multivariate Anal.\/}~{\em 100\/}(5), 913--935.

\bibitem[\protect\citeauthoryear{Brodie, Daubechies, De~Mol, Giannone, and
  Loris}{Brodie et~al.}{2009}]{BrodieEtAl2009}
Brodie, J., I.~Daubechies, C.~De~Mol, D.~Giannone, and I.~Loris (2009).
\newblock Sparse and stable {M}arkowitz portfolios.
\newblock {\em Proceedings the National Academy of Sciences of the United
  States of America\/}~{\em 106\/}(30), 12267--12272.

\bibitem[\protect\citeauthoryear{Chan, Horv{\'a}th, and Hu{\v{s}}kov{\'a}}{Chan
  et~al.}{2013}]{ChanHorvathHuskova2013}
Chan, J., L.~Horv{\'a}th, and M.~Hu{\v{s}}kov{\'a} (2013).
\newblock Darling-{E}rd{\H o}s limit results for change-point detection in
  panel data.
\newblock {\em J. Statist. Plann. Inference\/}~{\em 143\/}(5), 955--970.

\bibitem[\protect\citeauthoryear{Chen, Xu, and Wu}{Chen
  et~al.}{2013}]{ChenXuWu2014}
Chen, X., M.~Xu, and W.~B. Wu (2013).
\newblock Covariance and precision matrix estimation for high-dimensional time
  series.
\newblock {\em Ann. Statist.\/}~{\em 41\/}(6), 2994--3021.

\bibitem[\protect\citeauthoryear{Chu, Stinchcombe, and White}{Chu
  et~al.}{1996}]{ChuStinchcombeWhite1996}
Chu, C.-S.~J., M.~Stinchcombe, and H.~White (1996).
\newblock Monitoring structural change.
\newblock {\em Econometrica\/}~{\em 64\/}(5), 1045--1065.

\bibitem[\protect\citeauthoryear{Cs{\"o}rg{\H{o}} and
  Horv{\'a}th}{Cs{\"o}rg{\H{o}} and Horv{\'a}th}{1997}]{CsoergoeHorvath1997}
Cs{\"o}rg{\H{o}}, M. and L.~Horv{\'a}th (1997).
\newblock {\em Limit theorems in change-point analysis}.
\newblock Wiley Series in Probability and Statistics. John Wiley \& Sons, Ltd.,
  Chichester.
\newblock With a foreword by David Kendall.

\bibitem[\protect\citeauthoryear{Diaconis and Freedman}{Diaconis and
  Freedman}{1984}]{DiaconisFreedman1984}
Diaconis, P. and D.~Freedman (1984).
\newblock Asymptotics of graphical projection pursuit.
\newblock {\em Ann. Statist.\/}~{\em 12\/}(3), 793--815.

\bibitem[\protect\citeauthoryear{Fiecas, Franke, Sachs, and Tadjuidje}{Fiecas
  et~al.}{2014}]{FiecasFrankeSachs2014}
Fiecas, M., J.~Franke, R.~v. Sachs, and J.~Tadjuidje (2014).
\newblock Shrinkage estimation for multivariate hidden {M}arkov mixture models.
\newblock {\em Universit\'e catholique de Louvain, ISBA Discussion Paper
  2012/16; submitted and in revision\/}.

\bibitem[\protect\citeauthoryear{Fiecas and von Sachs}{Fiecas and von
  Sachs}{2014}]{FiecasSachs2013}
Fiecas, M. and R.~von Sachs (2014).
\newblock Data-driven shrinkage of the spectral density matrix of a
  high-dimensional time series.
\newblock {\em Electron. J. Stat.\/}~{\em 8\/}(2), 2975--3003.

\bibitem[\protect\citeauthoryear{Hosking}{Hosking}{1996}]{Hosking1996}
Hosking, J. R.~M. (1996).
\newblock Asymptotic distributions of the sample mean, autocovariances, and
  autocorrelations of long-memory time series.
\newblock {\em J. Econometrics\/}~{\em 73\/}(1), 261--284.

\bibitem[\protect\citeauthoryear{Hu{\v{s}}kov{\'a} and
  Hl{\'a}vka}{Hu{\v{s}}kov{\'a} and Hl{\'a}vka}{2012}]{HuskovaHlavka2012}
Hu{\v{s}}kov{\'a}, M. and Z.~Hl{\'a}vka (2012).
\newblock Nonparametric sequential monitoring.
\newblock {\em Sequential Anal.\/}~{\em 31\/}(3), 278--296.

\bibitem[\protect\citeauthoryear{Jagannathan and Ma}{Jagannathan and
  Ma}{2003}]{Jagannathan2003}
Jagannathan, R. and T.~Ma (2003).
\newblock Risk reduction in large portfolios: Why imposing the wrong
  constraints helps.
\newblock {\em Journal of Finance\/}~{\em LVIII}, 1651--1683.

\bibitem[\protect\citeauthoryear{Jaru{\v{s}}kov{\'a}}{Jaru{\v{s}}kov{\'a}}{2013}]{Jaruskova2013}
Jaru{\v{s}}kov{\'a}, D. (2013).
\newblock Testing for a change in covariance operator.
\newblock {\em J. Statist. Plann. Inference\/}~{\em 143\/}(9), 1500--1511.

\bibitem[\protect\citeauthoryear{Jirak}{Jirak}{2011}]{Jirak2011}
Jirak, M. (2011).
\newblock On the maximum of covariance estimators.
\newblock {\em J. Multivariate Anal.\/}~{\em 102\/}(6), 1032--1046.

\bibitem[\protect\citeauthoryear{Jirak}{Jirak}{2012}]{Jirak2012}
Jirak, M. (2012).
\newblock Change-point analysis in increasing dimension.
\newblock {\em J. Multivariate Anal.\/}~{\em 111}, 136--159.

\bibitem[\protect\citeauthoryear{Johnstone and Lu}{Johnstone and
  Lu}{2009}]{JohnstoneLu2009}
Johnstone, I.~M. and A.~Y. Lu (2009).
\newblock On consistency and sparsity for principal components analysis in high
  dimensions.
\newblock {\em J. Amer. Statist. Assoc.\/}~{\em 104\/}(486), 682--693.

\bibitem[\protect\citeauthoryear{Jolliffe, Trendafilov, and Uddin}{Jolliffe
  et~al.}{2003}]{JollifeEtAl2003}
Jolliffe, L., N.~Trendafilov, and M.~Uddin (2003).
\newblock A modified principal component technique based on the lasso.
\newblock {\em Journal of Computational and Graphical Statistics\/}~{\em 12},
  531--547.

\bibitem[\protect\citeauthoryear{Koml{\'o}s, Major, and Tusn{\'a}dy}{Koml{\'o}s
  et~al.}{1975}]{KMT1975}
Koml{\'o}s, J., P.~Major, and G.~Tusn{\'a}dy (1975).
\newblock An approximation of partial sums of independent {${\rm RV}$}'s and
  the sample {${\rm DF}$}. {I}.
\newblock {\em Z. Wahrscheinlichkeitstheorie und Verw. Gebiete\/}~{\em 32},
  111--131.

\bibitem[\protect\citeauthoryear{Koml{\'o}s, Major, and Tusn{\'a}dy}{Koml{\'o}s
  et~al.}{1976}]{KMT1976}
Koml{\'o}s, J., P.~Major, and G.~Tusn{\'a}dy (1976).
\newblock An approximation of partial sums of independent {RV}'s, and the
  sample {DF}. {II}.
\newblock {\em Z. Wahrscheinlichkeitstheorie und Verw. Gebiete\/}~{\em
  34\/}(1), 33--58.

\bibitem[\protect\citeauthoryear{Kouritzin}{Kouritzin}{1995}]{Kouritzin1995}
Kouritzin, M.~A. (1995).
\newblock Strong approximation for cross-covariances of linear variables with
  long-range dependence.
\newblock {\em Stochastic Process. Appl.\/}~{\em 60\/}(2), 343--353.

\bibitem[\protect\citeauthoryear{Ledoit and Wolf}{Ledoit and
  Wolf}{2003}]{LedoitWolf2003}
Ledoit, O. and M.~Wolf (2003).
\newblock Improved estimation of the covariance matrix of stock returns with an
  application to portfolio selection.
\newblock {\em Journal of Empirical Finance\/}~{\em 10}, 603--621.

\bibitem[\protect\citeauthoryear{Ledoit and Wolf}{Ledoit and
  Wolf}{2004}]{LedoitWolf2004}
Ledoit, O. and M.~Wolf (2004).
\newblock A well-conditioned estimator for large-dimensional covariance
  matrices.
\newblock {\em J. Multivariate Anal.\/}~{\em 88\/}(2), 365--411.

\bibitem[\protect\citeauthoryear{Markowitz}{Markowitz}{1952}]{Markowitz1952}
Markowitz, H. (1952).
\newblock Portfolio selection.
\newblock {\em Journal of Finance\/}~{\em 7}, 77--91.

\bibitem[\protect\citeauthoryear{Newey and West}{Newey and
  West}{1987}]{NeweyWest1987}
Newey, W.~K. and K.~D. West (1987).
\newblock A simple, positive semidefinite, heteroskedasticity and
  autocorrelation consistent covariance matrix.
\newblock {\em Econometrica\/}~{\em 55\/}(3), 703--708.

\bibitem[\protect\citeauthoryear{Philipp}{Philipp}{1986}]{Philipp1986}
Philipp, W. (1986).
\newblock A note on the almost sure approximation of weakly dependent random
  variables.
\newblock {\em Monatsh. Math.\/}~{\em 102\/}(3), 227--236.

\bibitem[\protect\citeauthoryear{Sancetta}{Sancetta}{2008}]{Sancetta2008}
Sancetta, A. (2008).
\newblock Sample covariance shrinkage for high dimensional dependent data.
\newblock {\em J. Multivariate Anal.\/}~{\em 99\/}(5), 949--967.

\bibitem[\protect\citeauthoryear{Shen and Huang}{Shen and
  Huang}{2008}]{ShenHuan2008}
Shen, H. and J.~Z. Huang (2008).
\newblock Sparse principal component analysis via regularized low rank matrix
  approximation.
\newblock {\em J. Multivariate Anal.\/}~{\em 99\/}(6), 1015--1034.

\bibitem[\protect\citeauthoryear{Shorack}{Shorack}{2000}]{Shorack2000}
Shorack, G.~R. (2000).
\newblock {\em Probability for statisticians}.
\newblock Springer Texts in Statistics. Springer-Verlag, New York.

\bibitem[\protect\citeauthoryear{Stein}{Stein}{1956}]{Stein1956}
Stein, C. (1956).
\newblock Inadmissibility of the usual estimator for the mean of a multivariate
  normal distribution.
\newblock In {\em Proceedings of the {T}hird {B}erkeley {S}ymposium on
  {M}athematical {S}tatistics and {P}robability, 1954--1955, vol. {I}}, pp.\
  197--206. University of California Press, Berkeley and Los Angeles.

\bibitem[\protect\citeauthoryear{Steland}{Steland}{2004}]{Steland2004}
Steland, A. (2004).
\newblock Sequential control of time series by functionals of kernel-weighted
  empirical processes under local alternatives.
\newblock {\em Metrika\/}~{\em 60\/}(3), 229--249.

\bibitem[\protect\citeauthoryear{Steland}{Steland}{2005}]{Steland2005}
Steland, A. (2005).
\newblock Optimal sequential kernel detection for dependent processes.
\newblock {\em J. Statist. Plann. Inference\/}~{\em 132\/}(1-2), 131--147.

\bibitem[\protect\citeauthoryear{Steland}{Steland}{2010}]{Steland2010}
Steland, A. (2010).
\newblock A surveillance procedure for random walks based on local linear
  estimation.
\newblock {\em J. Nonparametr. Stat.\/}~{\em 22\/}(3-4), 345--361.

\bibitem[\protect\citeauthoryear{Steland}{Steland}{2012}]{Steland2012}
Steland, A. (2012).
\newblock {\em Financial Statistics and Mathematical Finance}.
\newblock John Wiley \& Sons, Ltd., Chichester.

\bibitem[\protect\citeauthoryear{Steland}{Steland}{2015}]{Steland2015}
Steland, A. (2015).
\newblock Asymptotics for random functions moderated by dependent noise.
\newblock {\em Stat Inference Stoch Process\/}, 1--25.

\bibitem[\protect\citeauthoryear{Steland and Rafaj{\l}owicz}{Steland and
  Rafaj{\l}owicz}{2014}]{StelandRafalowicz2014}
Steland, A. and E.~Rafaj{\l}owicz (2014).
\newblock Decoupling change-point detection based on characteristic functions:
  methodology, asymptotics, subsampling and application.
\newblock {\em J. Statist. Plann. Inference\/}~{\em 145}, 49--73.

\bibitem[\protect\citeauthoryear{Tibshirani}{Tibshirani}{1996}]{Tibshirani1996}
Tibshirani, R. (1996).
\newblock Regression shrinkage and selection via the lasso.
\newblock {\em J. Roy. Statist. Soc. Ser. B\/}~{\em 58\/}(1), 267--288.

\bibitem[\protect\citeauthoryear{Tibshirani}{Tibshirani}{2011}]{Tibshirani2011}
Tibshirani, R. (2011).
\newblock Regression shrinkage and selection via the lasso: a retrospective.
\newblock {\em J. R. Stat. Soc. Ser. B Stat. Methodol.\/}~{\em 73\/}(3),
  273--282.

\bibitem[\protect\citeauthoryear{Witten and Tibshirani}{Witten and
  Tibshirani}{2008}]{WittenTibshirani2009}
Witten, D.~M. and R.~Tibshirani (2008).
\newblock Testing significance of features by lassoed principal components.
\newblock {\em Ann. Appl. Stat.\/}~{\em 2\/}(3), 986--1012.

\bibitem[\protect\citeauthoryear{Witten, Tibshirani, and Hastie}{Witten
  et~al.}{2009}]{TibshiraniWittenHastie2009}
Witten, D.~M., R.~Tibshirani, and T.~Hastie (2009).
\newblock A penalized decomposition, with applications to sparse principal
  components and canonical correlation analysis.
\newblock {\em Biostatistics\/}~{\em 10}, 515--534.

\bibitem[\protect\citeauthoryear{Wu}{Wu}{2007}]{Wu2007}
Wu, W.~B. (2007).
\newblock Strong invariance principles for dependent random variables.
\newblock {\em Ann. Probab.\/}~{\em 35\/}(6), 2294--2320.

\bibitem[\protect\citeauthoryear{Wu}{Wu}{2009}]{Wu2009}
Wu, W.~B. (2009).
\newblock An asymptotic theory for sample covariances of {B}ernoulli shifts.
\newblock {\em Stochastic Process. Appl.\/}~{\em 119\/}(2), 453--467.

\bibitem[\protect\citeauthoryear{Wu, Huang, and Zheng}{Wu
  et~al.}{2010}]{WuHuangZheng2010}
Wu, W.~B., Y.~Huang, and W.~Zheng (2010).
\newblock Covariances estimation for long-memory processes.
\newblock {\em Adv. in Appl. Probab.\/}~{\em 42\/}(1), 137--157.

\bibitem[\protect\citeauthoryear{Wu and Min}{Wu and Min}{2005}]{WuMin2005}
Wu, W.~B. and W.~Min (2005).
\newblock On linear processes with dependent innovations.
\newblock {\em Stochastic Process. Appl.\/}~{\em 115\/}(6), 939--958.

\bibitem[\protect\citeauthoryear{Wu and Xiao}{Wu and Xiao}{2011}]{Wu2011}
Wu, W.~B. and H.~Xiao (2011).
\newblock {\em Covariance matrix estimation in time series}, Volume~30 of {\em
  Handbook of Statistics}.
\newblock Elsevier.

\bibitem[\protect\citeauthoryear{Xiao and Wu}{Xiao and Wu}{2014}]{XiaoWu2014}
Xiao, H. and W.~B. Wu (2014).
\newblock Portmanteau test and simultaneous inference for serial covariances.
\newblock {\em Statist. Sinica\/}~{\em 24\/}(2), 577--599.

\end{thebibliography}

\end{document}